\documentclass[a4paper,reqno, 12pt]{amsart}
\author{Julia Brandes}
\title[Rational lines on hypersurfaces]{The density of rational lines on hypersurfaces: \\A bihomogeneous perspective}
\address{Department for Mathematical Sciences, University of Gothenburg and Chalmers University of Technology, 412 96 G\"oteborg, Sweden}
\email{brjulia@chalmers.se}
\subjclass[2010]{11D72, 11P55, 11E76, 14G05}
\keywords{Forms in many variables, Hardy-Littlewood method, rational lines}
\date{\today}

\usepackage[T1]{fontenc}
\usepackage[latin1]{inputenc}
\usepackage{latexsym}
\usepackage[arrow, matrix, curve]{xy}
\usepackage[dvips]{graphicx}
\usepackage{epsfig}
\usepackage{bm}
\usepackage{hhline}
\usepackage{amssymb}
\usepackage{amsmath}
\usepackage{verbatim}

\usepackage{dsfont}
\usepackage{amsthm}
\usepackage{enumerate}
\usepackage{mathrsfs}
\usepackage[hmargin=2.5cm,vmargin=3.5cm]{geometry}

\newtheorem{thm}{Theorem}
\newtheorem{lem}{Lemma}

\newtheorem{prop}[lem]{Proposition}

\numberwithin{equation}{section}
\numberwithin{thm}{section}
\numberwithin{lem}{section}	

\theoremstyle{definition}

\def\A{\mathbb A}
\def\C{\mathbb C}
\def\F{\mathbb F}

\def\Q{\mathbb Q}
\def\R{\mathbb R}

\def\Z{\mathbb Z}

\def\fr#1{\mathfrak{#1}}
\def\cal#1{\mathcal{#1}}
\def\scr#1{\mathscr{#1}}
\def\B#1{\mathbf{#1}}

\def\ba{\bm{\alpha}}
\def\bb{\bm{\beta}}
\def\bg{\bm{\gamma}}

\def\D{\,\mathrm{d}}
\def\dsum#1#2{\sum_{\substack{{#1}\\{#2}}}}

\def\a0{\alpha_0}

\def\eps{\varepsilon}
\def\mmod#1{\;(\mathrm{mod}\;{#1})}

\renewcommand\le{\leqslant}
\renewcommand\ge{\geqslant}

\DeclareMathOperator{\rk}{rank}
\DeclareMathOperator{\card}{Card}

\DeclareMathOperator{\vol}{vol}

\newenvironment{pf}{\begin{proof}[Proof]}{\end{proof}}

\usepackage{hyperref}

\begin{document}

\begin{abstract}
    Let $F$ be a non-singular homogeneous polynomial of degree $d$ in $n$ variables. We give an asymptotic formula of the pairs of integer points $(\B x, \B y)$ with $|\B x| \le X$ and $|\B y| \le Y$ which generate a line lying in the hypersurface defined by $F$, provided that $n > 2^{d-1}d^4(d+1)(d+2)$. In particular, by restricting to Zariski-open subsets we are able to avoid imposing any conditions on the relative sizes of $X$ and $Y$. 
\end{abstract}

\maketitle

\section{Introduction}

Questions concerning the number and distribution of rational points on hypersurfaces have long attracted the interest of both number theorists and algebraic geometers. Building on work by Davenport \cite{dav32}, Birch wrote an influential paper \cite{birch} in which he provided a method to prove the analytic Hasse principle and establish asymptotic formul{\ae} for the number of integer points on projective hypersurfaces under moderate non-singularity conditions, provided that the dimension of the hypersurface is sufficiently large compared to its degree. In particular, suppose that $F \in \Z[x_1, \ldots, x_n]$ is a non-singular form of degree $d$ defining a hypersurface $\cal V$, and write $N(X)$ for the number of points $\B x \in \cal V(\Z)$ with $|x_i| \le X$ for $1 \le i \le n$. In this notation, Birch's main result \cite[Theorem]{birch} states that whenever $n > 2^d (d-1)$, there exists a positive real number $\nu$ with the property that the number of integer points on $\cal V$ satisfies an asymptotic formula of the shape
\begin{align*}
	N(X) = c X^{n-d} + O(X^{n-d-\nu}).
\end{align*}
The constant $c$ is non-negative and has an interpretation in terms of the density of $K_v$-points in $\cal V$ for all completions $K_v$ of $\Q$. 

In the work at hand, we study a higher-dimensional generalisation of Birch's result. Denote by $N(X,Y)$ the number of points $\B x, \B y \in \Z^n \setminus \{\bm 0\}$ satisfying $|x_i| \le X$ and $|y_i| \le Y$ for $1 \le i \le n$, and having the property that 
\begin{align}\label{1.1}
	F(u \B x + v \B y) = 0 \qquad \text{identically in $u$ and $v$.}
\end{align}
This problem is related to that of counting rational lines contained in $\cal V$, in that it counts all possible sets of generating pairs $(\B x, \B y)$ of suitably bounded height and with the property that the line spanned by $(\B x, \B y)$ is fully contained in $\cal V$. Geometrically, it is known that the Fano scheme of lines on a generic hypersurface $\cal V$ of degree $d$ has dimension $2n-d-5$ whenever that number is positive (see e.g. Langer \cite{langer}). When $F$ is a cubic form, recent work of the author jointly with Dietmann \cite{JB-RD} shows that the equation \eqref{1.1} has non-trivial rational solutions whenever $n \ge 29$, but that there may not be any rational solutions when $n=11$ or lower. For more general settings, the equation \eqref{1.1} has been investigated in a series of papers by the present author \cite{FRF,FRF2,FRF3,FRF4}. 
We note at this point that, in order to strictly count lines, we would have to exclude those solutions of \eqref{1.1} where $\B x$ and $\B y$ are proportional. Fortunately, the contribution of such points is of a smaller order of magnitude than our eventual main term, so we do not lose any generality by omitting to explicitly exclude them. 

A special role in problems of this flavour is played by certain points $\B y \in \cal V$ that admit for a disproportionate number of solutions $\B x \in \cal V$ satisfying \eqref{1.1}. Typically, the contribution arising from these solutions is counterbalanced by the relative sparsity of such points $\B y$, but when $Y$ is very small in comparison to $X$, such solutions might well dominate the overall count. It is therefore natural to exclude the solutions that arise from such special subvarieties. When $\cal U \subseteq \cal V$ is a Zariski-open subset, we denote by $N_{\cal U}(X,Y)$ the number of integral $\B x, \B y \in \cal U$ with $|x_i| \le X$ and $|y_i| \le Y$ for $1 \le i \le n$ that satisfy \eqref{1.1}. We can now state the main result of this memoir. 

\begin{thm}\label{T1.1}
	Let $F \in \Z[x_1, \ldots, x_n]$ be a non-singular form of degree $d \ge 5$ defining a hypersurface $\cal V$. Let further 
	\begin{align*}
		n > 2^{d-1}d^4(d+1)(d+2).
	\end{align*}
	Then there exists a Zariski-open subset $\cal U \subseteq \cal V$ and a positive real number $\nu$ with the property that 
	\begin{align*}
		N_{\cal U}(X,Y) = (XY)^{n-\frac12 d(d+1)} \chi_{\infty} \prod_{p \text{ prime}} \chi_p + O((XY)^{n-\frac12d(d+1) - \nu}).
	\end{align*}
	The Euler product converges absolutely, and its factors have an interpretation as the density of solutions of \eqref{1.1} over the local fields $\R$ and $\Q_p$, respectively. 
\end{thm}
Note that Theorem~\ref{T1.1} is a slightly simplified version of what our methods yield; by a more thorough analysis it would be possible to obtain some improvements in the lower-order terms at the expense of a significantly more complicated expression, but no easy improvement of the order of growth $2^d d^6$ in our result. In particular, we do not expect our results to be competitive when $d$ is small. For this reason, even though a modification of our approach would provide results for $d\in\{2,3,4\}$ also, we refrain from including the analysis of those cases as the expected results would likely be quite weak.

Clearly, the problem is symmetric in $X$ and $Y$, so in our discussion we may assume without loss of generality that $Y \le X$. In the special case when $Y=X$, the conclusion of Theorem~\ref{T1.1} follows from \cite[Theorem~1.1]{FRF} under the more lenient condition that $n>3 \cdot 2^d (d-1)(d+2)$, and subsequent work \cite[Theorem~2.1]{FRF2} establishes a conclusion similar to that of Theorem~\ref{T1.1} above under the additional condition that $n$ should be large enough in terms of $\log X/\log Y$, which is acceptable if $X$ is at most a bounded power of $Y$. 
The main new input in our present work is therefore our treatment of the situation when $Y$ is vastly smaller than $X$. Unlike in our former work in \cite{FRF,FRF2}, where we allowed the variables $\B x$, $\B y$ to vary independently, we pursue a slicing approach inspired by \cite{damaris} in which we fix a point $\B y \in \cal U(\Z)$ and then investigate the number $N_{\B y}(X; \cal U)$ of points $\B x \in \cal U(\Z) \cap [-X,X]^n$ for which \eqref{1.1} is satisfied with that particular value $\B y$. We then have 
\begin{align}\label{1.2}
	N_{\cal U}(X,Y) = \sum_{\substack{\B y \in \cal U(\Z)\\ |\B y| \le Y}} N_{\B y}(X; \cal U),
\end{align}
and we aim to establish bounds of the shape 
\begin{align*}
	N_{\B y}(X; \cal U) = c_{\B y} X^{n-\frac12 d(d+1)} + O(X^{n-{\frac12} d(d+1)-\nu})
\end{align*} 
for some constant $c_{\B y}$ and some positive number $\nu$. 

For generic $\B y$, the quantity $N_{\B y}(X) = N_{\B y}(X; \cal V)$ can be understood by applying the methods of Browning and Heath-Brown \cite{bhb} for systems of homogeneous equations with differing degrees, although we need to be careful to track the dependence on the coefficients as these will be polynomially dependent on $\B y$. Unfortunately, this strategy breaks down if $\B y$ fails to satisfy a certain second-order non-singularity condition. When $H_{\B x}$ denotes the Hessian of $F$ at the point $\B x$, we set 
\begin{align*}
	\cal V^*_{2,\rho} = \{\B x \in \cal V: \rk H_{\B x} \le n-\rho\},
\end{align*}
and let $\cal V_{2,\rho} = \cal V \setminus \cal V^*_{2,\rho}$. In particular, $\cal V_{2,\rho}$ is Zariski-open in $\cal V$ for all $1 \le \rho \le n$. 

The following two by-products of our strategy may be of independent interest and are simplified versions of Theorems \ref{T6.1} and \ref{T6.2} below, respectively.

\begin{thm}\label{T1.2}
	Let $F \in \Z[x_1, \ldots, x_n]$ be a non-singular form of degree $d \ge 5$ defining a hypersurface $\cal V$. Let further $\psi \in (0,1/(2 d^4)]$, and suppose that 
	\begin{align*}
		n\ge 2^{d}d(d^2-1) + \rho.
	\end{align*}
	Then there exists a positive real number $\nu$ with the property that 
	\begin{align*}
		N_{\B y}(X) = X^{n-\frac12 d(d+1)} \fr S_{\B y} \fr J_{\B y} + O(X^{n-\frac12 d(d+1) - \nu}) 
	\end{align*}
	uniformly for all $\B y \in \cal V_{2,\rho}(\Z)$ satisfying $|\B y| \le X^{\psi}$. Moreover, the local factors satisfy $0 \le \fr S_{\B y} \ll_{\B y} 1$ and $0 \le \fr J_{\B y} \ll_{\B y} 1$.
\end{thm}

The set $\cal V^*_{2,\rho}$ is clearly algebraically defined for any $\rho$, and it is known (see e.g. \cite[Lemma~2]{HB10}) that $\dim \cal V^*_{2,\rho} \le n-\rho$. Consequently, we have $N_{\B y}(X; \cal V_{2,\rho})=N_{\B y}(X)+O(X^{n-\rho})$, and we see that when $\rho >\frac12 d(d+1)$, the anticipated main term exceeds any error that might arise if we replace $\cal V_{2,m}$ by $\cal V$ itself. This allows us to derive a bound on $N_{\cal U}(X,Y)$ from bounds on $N_{\B y}(X)$. 

\begin{thm}\label{T1.3}
	Let $F$ and $\cal V$ be as before with $d \ge 5$, and for some $\psi \in (0,1/(2d^4)]$ set $Y=X^{\psi}$. Suppose that 
	\begin{align*}
		n \ge  2^d d(d^2-1)
	\end{align*}
	and set $\cal U = \cal V_{2, \frac12 d(d+1)+1}$. 
	Then there exists a real number $\nu > 0$ for which 
	\begin{align*}
		N_{\cal U}(X,Y) = X^{n-\frac12 d(d+1)} \sum_{\substack{\B y \in \cal U(\Z) \\|\B y| \le Y}} \fr S_{\B y} \fr J_{\B y} + O((XY)^{n-\frac12 d(d+1)-\nu}).
	\end{align*}
\end{thm}

The set $\cal V^*_{2,1}$ is, as mentioned above, a subvariety of $\cal V$ with codimension $1$. In particular, it is defined by the zero set of the simultaneous equations $F(\B x) = 0$ and $\det H_{\B x}=0$. The function $\Delta(\B x)=\det H_{\B x}$ is a form of degree $(d-2)n$ in $n$ variables, and according to standard heuristics one might hope that, unless the variety defined by $\Delta(\B x)=0$ contains high-dimensional subvarieties of low degree, the set $\cal V^*_{2,1}(\Z)$ might only have a finite number of primitive points, and might potentially even consist only of the origin. In such a situation, it would be permissible in Theorem~\ref{T1.1} to take $\cal U = \cal V \setminus \{ \bm 0\}$. Unfortunately, our current understanding of the size of the set $\cal V^*_{2,\rho}$ is quite weak. Not  only is there no sufficient condition on the geometry of $F$ presently known that would imply that $\cal V^*_{2,\rho}(\Z) = \{ \bm 0\}$  for some $\rho$ sufficiently small compared to $n$, but indeed such a result seems to be quite out of reach for present methods. 
Nonetheless, by bounding the number of integral points in $\cal V^*_{2,\rho}$ we are still able to establish asymptotic formul{\ae} for $N(X,Y)$  that extend the admissible range of $Y$ compared to what had previously been known in \cite[Theorem~2.1]{FRF2}.

\begin{thm}\label{T1.4}
	Let $F \in \Z[x_1, \ldots, x_n]$ be a non-singular form of degree $d \ge 5$, and suppose that $Y = X^\psi$, where $0 < \psi < (2d^4)^{-1}$. Furthermore, suppose that
	\begin{align*}
		n> 2^{d-1}d^4(d+1)(d+2) + \textstyle{\frac12}d(d-1) \psi^{-1}.
	\end{align*}
	Then there exists a positive real number $\nu$ with the property that
	\begin{align*}
		N(X, Y) = (XY)^{n-\frac12d(d+1)} \chi_{\infty}\prod_{p \text{ prime}} \chi_p + O((XY)^{n-\frac12d(d+1)}X^{-\nu}),
	\end{align*}
	where the local factors are the same as in Theorem~\ref{T1.1}.
\end{thm}

The reader may wonder how the lower bound on $n$ compares with that which can be extracted from \cite[Theorem~2.1]{FRF2}. In that result, the bound on the number of variables in the case when $\psi$ is small can be written in terms of $\psi$ as 
\begin{align*}
	n > 2^{d-2}  d(d+1)(1+\psi^{-1}). 
\end{align*}
It is clear that for $\psi \ll d^{-4}$ our new result is significantly stronger.

\textbf{Notation.} Throughout the paper, the following notational conventions will be observed. Any  statements containing the letter $\eps$ are asserted to hold for all sufficiently small values of $\eps$, and we make no effort to track the precise `value' of $\eps$, which is consequently allowed to change from one line to the next. We will be liberal in our use of vector notation. In particular, equations and inequalities involving vectors should always be understood entrywise. In this spirit, we write $| \B x | = \| \B x \|_\infty = \max |x_i|$, as well as $(\B a, b) = \gcd(a_1, \ldots, a_n, b)$. For $\alpha \in \R$ we write $\| \alpha \| = \min_{z \in \Z}|\alpha - z|$. Finally, the implicit constants in the Landau and Vinogradov notations are allowed to depend on all parameters except $X$, $Y$ and $\B y$.

\textbf{Acknowledgments.} During the production of this memoir, the author was supported by Starting Grant no.~2017-05110 from the Swedish Research Council (Vetenskapsr{\aa}det). The author is also grateful to Tim Browning and Per Salberger for valuable conversations around the topic of this paper.

\section{Van der Corput differences}\label{S2}

Let $\Phi$ denote the symmetric $d$-linear form associated to $F$, so that $F(\B x) = \Phi(\B x, \ldots, \B x)$. Then after expanding, the form $F$ may be written as
\begin{equation*}
	F\left(u \B x+ v \B y\right) = \sum_{j=0}^d \binom{d}{j} u^j v^{d-j} \Phi(\underbrace{\B x, \ldots, \B x}_{j \text{ entries}}, \underbrace{\B y, \ldots, \B y}_{d-j \text{ entries}}),
\end{equation*} 
and our counting function $N_{\cal U}(X, Y)$ counts integer solutions $\B x, \B y \in \cal U$  to the system of equations
\begin{align}\label{2.1}
	\Phi(\underbrace{\B x, \ldots, \B x}_{j \text{ entries}}, \underbrace{\B y, \ldots, \B y}_{d-j \text{ entries}}) = 0 \qquad (0 \le j \le d),
\end{align}
where $|x_i| \le X$ and $|y_i|\le Y$ for $1 \le i \le n$.

In this and the following sections we fix a value of $\B y$ and consider \eqref{2.1} as a system of equations in $\B x$ only. Eventually, we will have to consider only such choices for $\B y$ that lie in a suitable Zariski-open subset $\cal U$. This allows us in particular to exclude the value $\B y = \bm 0$. For $1 \le j \le d$ we write $\Phi^{(j)}_{\B y}(\B x)$ for the form having $j$ entries $\B x$ and $d-j$ entries $\B y$. In this notation, $N_{\B y}(X)$ denotes the number of points $\B x \in \Z^n \cap [-X, X]^n$ satisfying 
\begin{align}\label{2.2}
	\Phi_{\B y}^{(j)}(\B x) &= 0 \qquad (1 \le j \le d).
\end{align}

The system \eqref{2.2} consists of forms of consecutive degrees $1, \ldots, d$. Asymptotic formul{\ae} for the number of solutions of such systems can be obtained by the machinery of Browning and Heath-Brown \cite{bhb}. However, before embarking on that argument, it is convenient to eliminate one variable by solving the linear equation, so that all forms explicitly occurring in the system have degree two or higher. To this end, observe that the equation $\Phi_{\B y}^{(1)}(\B x) = 0$ can be expressed as 
\begin{align}\label{2.3}
	l_1(\B y) x_1 + \ldots + l_n(\B y) x_n = 0, 
\end{align}
where the coefficients $l_i = l_i(\B y)$ are polynomials of degree $d-1$ in $\B y$.  Since $F$ is non-singular by assumption, the set of $\B x \in \Z^n$ satisfying \eqref{2.3} forms an $(n-1)$-dimensional lattice $\Lambda_{\B y} \subseteq \Z^n$. Denote by $\fr A_{\B y}(X) \subseteq \Z^n$ the set of lattice points $\B x \in \Lambda_{\B y}$ for which $|\B x| \le X$. Thus, we may equivalently consider the quantity $N_{\B y}(X)$ to be given by the number of points $\B x \in \fr A_{\B y}(X)$ satisfying the system of equations 
\begin{align*}
	\Phi_{\B y}^{(j)}(\B x) &= 0 \qquad (2 \le j \le d).
\end{align*}

In order to understand the counting function $N_{\B y}(X)$, we encode the summation conditions in exponential sums. Let $\ba=(\alpha_2, \ldots, \alpha_{d}) \in [0,1)^{d-1}$, then $N_{\B y}(X)$ is given by
\begin{align}\label{2.5}
    N_{\B y}(X) &= \sum_{\B x \in \fr A_{\B y}(X)} \int_{[0,1)^{d-1}} e\bigg(\sum_{j=2}^d \alpha_j \Phi^{(j)}_{\B y}(\B x)\bigg) \D \ba
    = \int_{[0,1)^{d-1}}T_{\B y}(\ba; X) \D \ba,\
\end{align}
where we introduced the exponential sum
\begin{align*}
    T_{\B y}(\ba; P)=\sum_{\B x \in \fr A_{\B y}(P)}e\bigg(\sum_{j=2}^d \alpha_j \Phi^{(j)}_{\B y}(\B x)\bigg).
\end{align*}
In our arguments below, we will omit the parameter $P$ from the notation whenever there is no danger of confusion. In particular, we drop it in  most cases when $P=X$, highlighting it only when we consider exponential sums of size different from $X$. 

For simpler notation below, we write $s = n-1$.  By \cite[Lemma 1]{HB02}, the lattice $\Lambda_{\B y}$ has discriminant
\begin{align}\label{2.4}
	d(\Lambda_{\B y})\asymp | \B l(\B y)| \ll |\B y|^{d-1},
\end{align}
and we have $\card \fr A_{\B y}(X) \asymp X^{s}/d(\Lambda_{\B y})$. Fix a basis $\cal B = \{\B b_1, \ldots, \B b_s\} \subseteq \R^n$ of $\Lambda_{\B y}$, which by the same lemma we are free to choose in such a way that $|\B b_i| \asymp \mu_i$, where $\mu_1 > \ldots, > \mu_s$ are the successive minima of the lattice $\Lambda_{\B y}$. Thus, when $\B x \in \fr A_{\B y}(X)$ with $\B x =\xi_1 \B b_1 + \ldots + \xi_s \B b_s$, we have $\xi_i \ll X/\mu_i$ for $1 \le i \le s$. It is known that $\mu_1 \cdots \mu_s \asymp d (\Lambda_{\B y})$. Set 
\begin{align*}
	\fr B_{\B y}(X) = \prod_{i=1}^s [-c X/\mu_i, c X/\mu_i],
\end{align*}
where $c \ll 1$ is chosen large enough so that the coordinate vector $\bm \xi$ of $\B x$ lies in $\fr B_{\B y}(X)$ whenever $\B x \in \fr A_{\B y}(X)$. Moreover, for $2 \le j \le d$ set $\Psi^{(j)}_{\B y}(\bm \xi)=\Phi^{(j)}_{\B y}(\B x)$ and write
\begin{align*}
	\phi_{\B y}(\ba;\bm \xi) = \sum_{j=2}^d \alpha_j \Psi^{(j)}_{\B y}(\bm \xi).
\end{align*}

By an argument along the lines of that of Lemma~5.2 in \cite{FRF} one sees that 
\begin{align} \label{2.6}
 	T_{\B y}(\ba) \ll  X^{\eps}U_{\B y}(\ba), 
\end{align}
where 
\begin{align*}
	U_{\B y}(\ba)=\sup_{\bm \eta \in [0,1]^{d-1}}\Bigg|\sum_{\bm \xi \in \fr B_{\B y}(X)}e(\phi_{\B y}(\bm \alpha; \bm \xi) + \bm \eta \cdot \bm \xi)\Bigg|.
\end{align*}
This exponential sum is related to that considered by Schindler and Sofos \cite{SS} in their treatment of forms in many variables over lopsided boxes. In comparison with their result, however, our argument is more sensitive to the degree of the lopsidedness of the box. Fortunately, the discriminant of our lattice is fairly small. Indeed, since our methods will break down when $\psi \gg 1/d^2$ (see \eqref{2.10} below), and for our theorems we require even $\psi \ll 1/d^4$, we find ourselves in a situation where the discriminant of our lattice satisfies the bound $d(\Lambda_{\B y}) \ll X^{(d-1)\psi} \ll X^{O(1/d^3)}$.

The discrete differencing operator~$\partial$ is defined by its action on a polynomial~$F$ via the relation $\partial_{\B h}F(\B x) = F(\B x + \B h)-F(\B x)$, and we write
\begin{align*}
    \partial_{\B h_i, \ldots, \B h_1}F(\B x) = \partial_{\B h_i} \cdots \partial_{\B h_1} F(\B x)
\end{align*}
for its~$i$-fold iteration. This allows us to state our basic differencing lemma, which is fairly straightforward and essentially follows from~\cite[Lemma~2.1]{birch}.

\begin{lem}\label{L2.1}
    Let $1 \le i \le d-1$. Then one has
    \begin{align*}
        | U_{\B y}(\ba)|^{2^{i}} \ll \left(\frac{X^s}{d(\Lambda_{\B y})}\right)^{(2^{i}-i-1)} \dsum{\B h_l \in \fr B_{\B y}(X)}{1 \le l \le i} \Bigg|\sum_{\bm \xi \in \fr C}e\bigg(\partial_{\B h_i, \ldots, \B h_1} \phi_{\B y}(\ba;\bm \xi)\bigg)\Bigg|,
    \end{align*}
    where the sets $\fr C = \fr C(\B h_1, \ldots, \B h_i)$ are boxes (possibly empty) contained inside $\fr B_{\B y}(X)$.
\end{lem}

\begin{proof}
  	Upon recalling that $\card \fr B_{\B y}(X) \asymp X^s/d(\Lambda_{\B y})$, this is a straightforward reformulation of the standard Weyl differencing procedure as for instance in Davenport's monograph~\cite[Chapter 13]{dav}.
\end{proof}

At this stage, the usual procedure would be to apply Lemma~\ref{L2.1} with $i=d-1$, so that the argument of the exponential function becomes linear, thus yielding either a non-trivial upper bound or good approximations to  the coefficient $\alpha_d$. In the situation at hand, however, this approach would lose all information connected to the forms $\Psi^{(j)}_{\B y}$ with $j<d$. So instead we follow the approach by Browning and Heath-Brown \cite{bhb} and replace the last Weyl differencing step by a suitable van der Corput step. For $2 \le j \le d$ we define functions $B_{\B y,m}^{(j)}$ for $1 \le m \le s$ via the relation
\begin{align}\label{2.7}
  	\Psi_{\B y}^{(j)}(\bm \xi, \B h_1, \ldots, \B h_{j-1}) = \sum_{m=1}^s \xi_m B_{\B y,m}^{(j)}(\B h_1, \ldots, \B h_{j-1}).
\end{align}
Furthermore, let $\theta_2, \ldots, \theta_d$ be parameters in the unit interval which will be fixed later, and define
\begin{align}\label{2.8}
   \nu_j = (j-1) \theta_j  \qquad \text{and} \qquad \omega_j = \sum_{i=j}^{d} \nu_i \qquad (2 \le j \le d).
\end{align}
Set further $D_j =\frac12j(j+1)$ for $1 \le j \le d$, and for integers $q_j$ with $2 \le j \le d$ put
\begin{align}\label{2.9}
    Q_j = \prod_{i=j}^d q_i.
\end{align}
For notational reasons we write 
\begin{align*}
	D = D_d, \qquad D_0 = 0, \qquad \omega_{d+1} = 0 \qquad \text{and} \qquad Q_{d+1}=1,
\end{align*}
and we assume 
\begin{align}\label{2.10}
	\psi <1/(2d^2) 
\end{align}
throughout.
For fixed $\theta_{j+1}, \ldots, \theta_{d}$ set
\begin{align}\label{2.11}
	R_j = X^{1-\omega_{j+1}} |\B y|^{-D_{d-j}} \mu_1^{-(d-j)}
\end{align}
and
\begin{align}\label{2.12}
	\Upsilon_j=\dsum{\B h_l \in \fr B_{\B y} (X)}{1 \le l \le j-2}\sum_{\B w \in \fr B_{\B y}( 2R_j)} \prod_{m=1}^s \min\Bigg\{\frac{X}{\mu_m}, \bigg\| j!Q_{j+1}\alpha_{j} B_{\B y,m}^{(j)}(\B h_1, \ldots, \B h_{j-2}, \B w) \bigg\|^{-1} \Bigg\} .
\end{align}
We can now state one of our key iterative lemmas.
\begin{lem}\label{L2.2}
    Let $j \in \{2, \ldots, d\}$ be fixed. When $j < d$, suppose that $\theta_{j+1}, \ldots, \theta_{d}$ are fixed in such a way that, observing \eqref{2.8}, one has
    \begin{align}\label{2.13}
        \omega_{j+1} + \psi (D_{d-j}  + (d-1)(d-j))< 1.
    \end{align}
    Suppose that for any $i$ with $j < i \le d$ there exists a natural number $q_i \ll  X^{\nu_i}|\B y|^{d-i} \mu_1$ with the property that, in view of \eqref{2.9}, one has
    \begin{align*}
        \big\|Q_i \alpha_i\big\| \ll X^{-i + \omega_i}|\B y|^{D_{d-i}}\mu_1^{d-i+1}.
    \end{align*}
    Then we have the bound
    \begin{align*}
        |U_{\B y}(\ba)|^{2^{j-1}} \ll
        \left(\frac{X^s}{d(\Lambda_{\B y})}\right)^{2^{j-1}-(j-1)}\left(\frac{R_j^s}{d(\Lambda_{\B y})}\right)^{-1} \Upsilon_j.
    \end{align*}
\end{lem}

\begin{pf}
    Suppose first that $j > 2$. In this case, applying Lemma~\ref{L2.1} with $i=j-2$ followed by an application of Cauchy's inequality gives
    \begin{align}\label{2.14}
        | U_{\B y}(\ba) |^{2^{j-1}} \ll \left(\frac{X^s}{d(\Lambda_{\B y})}\right)^{2^{j-1}-j} \dsum{\B h_l \in \fr B_{\B y}(X)}{1 \le l \le j-2} \left|\sum_{\bm \xi \in \fr C_1} e(\partial_{\B h_1, \ldots, \B h_{j-2}} \phi_{\B y}(\ba;\bm \xi) )\right|^2
    \end{align}
    for suitable boxes $\fr C_1 = \fr C_1(\B h_1, \ldots, \B h_{j-2}) \subseteq \fr B_{\B y}(X)$.
    Let now $\B h_1, \ldots, \B h_{j-2}$ be temporarily fixed, and observe that our hypothesis concerning the size of the $q_i$ implies via \eqref{2.9} that $R_j Q_{j+1} \ll X$. Consequently, we have
    \begin{align}\label{2.15}
        &\frac{R_j^s}{d(\Lambda_{\B y})} \sum_{\bm \xi \in  \fr C_1} e(\partial_{\B h_1, \ldots, \B h_{j-2}} \phi_{\B y}(\ba;\bm \xi) )   \ll \sum_{\B u \in \fr B_{\B y}( R_j)}\sum_{\substack{\bm \xi\\ \bm \xi + Q_{j+1}\B u \in \fr C_1}} e(\partial_{\B h_1, \ldots, \B h_{j-2}} \phi_{\B y}(\ba;\bm \xi+ Q_{j+1} \B u) ).
    \end{align}
    
    We denote by $\fr C_2$ the set of $\bm \xi$ for which $\bm \xi + Q_{j+1}\B u \in \fr C_1$ for some $\B u \in \fr B_{\B y}(R_j)$; this box has cardinality $\card \fr C_2 \asymp X^s/d(\Lambda_{\B y})$. Then with another application of Cauchy's inequality one obtains from \eqref{2.15} the bound
    \begin{align*}
        &\left(\frac{R_j^{s}}{d(\Lambda_{\B y})}\right)^2 \left| \sum_{\bm \xi \in \fr C_1}e(\partial_{\B h_1, \ldots, \B h_{j-2}} \phi_{\B y}(\ba;\bm \xi))\right|^2 \\
        & \qquad \ll \frac{X^{s}}{d(\Lambda_{\B y})} \sum_{\bm \xi \in  \fr C_2} \Bigg| \sum_{\substack{ \B u \in \fr B_{\B y}( R_j)\\ \bm \xi + Q_{j+1}\B u \in \fr C_1}} e(\partial_{\B h_1, \ldots, \B h_{j-2}} \phi_{\B y}(\ba;\bm \xi + Q_{j+1} \B u) )\Bigg|^2 \\
        & \qquad \ll \frac{X^{s}}{d(\Lambda_{\B y})} \sum_{\B u, \B v \in \fr B_{\B y}( R_j) } \left| \sum_{\bm \xi}
         e\big(\partial_{\B h_1, \ldots, \B h_{j-2}} (\phi_{\B y}(\ba;\bm \xi + Q_{j+1} \B u) - \phi_{\B y}(\ba;\bm \xi + Q_{j+1} \B v))\big)\right|,
    \end{align*}
    where the inner sum runs over all $\bm \xi$ for which both $\bm \xi + Q_{j+1}\B u$ and $\bm \xi + Q_{j+1}\B v$ lie in $\fr C_1$.
    
    We now make the change of variables $\bm \xi' = \bm \xi + Q_{j+1}\B v$ and $\B w = \B u - \B v$, so that 
    \begin{align*}
        \phi_{\B y}(\ba;\bm \xi + Q_{j+1} \B u) -\phi_{\B y}(\ba;\bm \xi + Q_{j+1} \B v)  =  \partial_{Q_{j+1}\B w} \phi_{\B y}(\ba;\bm \xi').
    \end{align*}
    Thus, upon summing trivially over $\B v$, we have shown that
    \begin{align*}
        \left| \sum_{\bm \xi \in  \fr C_1} e( \partial_{\B h_1, \ldots, \B h_{j-2}}  \phi_{\B y}(\ba;\bm \xi)) \right|^2 
        & \ll \left(\frac{X}{R_j}\right)^s \sum_{\B w \in \fr B_{\B y}(2R_j)}  \sup_{\fr C \subseteq  \fr B_{\B y}(X)}\left|\sum_{\bm \xi \in \fr C} e(\partial_{\B h_1, \ldots, \B h_{j-2}, Q_{j+1}\B w} \phi_{\B y}(\ba;\bm \xi))\right|,
    \end{align*}
    where the supremum is over all coordinate-aligned boxes $\fr C$ inside $\fr B_{\B y}(X)$.   
    Thus, upon combining this bound with \eqref{2.14}, it follows that the exponential sum can be bounded above via 
    \begin{align*}
         | U_{\B y}(\ba) |^{2^{j-1}} \ll \left(\frac{X^s}{d(\Lambda_{\B y})}\right)^{2^{j-1}-(j-1)}  \left(\frac {R_j^s}{d(\Lambda_{\B y})}\right)^{-1} \cal W_j,
    \end{align*}
    where 
    \begin{align}\label{3.x}
    	\cal W_j = \dsum{\B h_l \in \fr B_{\B y}(X)}{1 \le l \le j-2} \sum_{\B w \in \fr B_{\B y}(2 R_j)} 
    	\sup_{\fr C \subseteq \fr B_{\B y}(X)} \Bigg| \sum_{\bm \xi \in \fr C} e\bigg(\partial_{\B h_1, \ldots, \B h_{j-2}, Q_{j+1}\B w} \sum_{i=2}^d \alpha_i \Psi_{\B y}^{(i)}(\bm \xi) \bigg)\Bigg|.
    \end{align}
	An analogous bound is also derived easily in the omitted case when $j=2$ upon interpreting the empty sum over $\B h_l$ and the concomitant differences as void, and noting that the phase factor in $U_{\B y}(\ba)$ disappears in the van der Corput step. 

    The size of the innermost exponential sum in \eqref{3.x} is dominated by the term corresponding to $i=j$. In fact, observe that after $j-1$ differences taken only the terms $\Psi_{\B y}^{(i)}(\bm \xi)$ with $ i \geq j$ occur explicitly in the argument of the exponential, and due to the last $Q_{j+1}$-van der Corput step all of these contain a factor $Q_{j+1}$. Hence whenever $j < d$ and $1 \le l \le s$ one has
    \begin{align*}
        &\frac{\partial}{\partial \xi_l} e\bigg(\partial_{\B h_1, \ldots, \B h_{j-2}, Q_{j+1}\B w} \sum_{i=j+1}^d \alpha_i \Psi_{\B y}^{(i)}(\bm \xi) \bigg) 
        \ll \sum_{i=j+1}^d \left\|Q_{j+1} \alpha_i\right\| X^{i-2} R_j |\B y|^{d-i} \mu_l \\
        &\ll \sum_{i=j+1}^d \left|\frac{Q_{j+1}}{Q_i}\right| \left\|Q_i\alpha_i\right\| X^{i-1-\omega_{j+1}}|\B y|^{-D_{d-j}+d-i } \mu_1^{-(d-j)}\mu_l 
        \ll X^{-1} \mu_l,
    \end{align*}
    where in the last step we used the hypotheses of the lemma. Upon iterating this procedure, one confirms for any subset $\{l_1, \ldots, l_k\} \subseteq \{1, \ldots, s\}$ that 
    \begin{align*}
		&\frac{\partial^k}{\partial \xi_{l_1} \cdots \partial \xi_{l_k}} e\bigg(\partial_{\B h_1, \ldots, \B h_{j-2}, Q_{j+1}\B w} \sum_{i=j+1}^d \alpha_i \Psi_{\B y}^{(i)}(\bm \xi) \bigg) \ll X^{-k} \mu_{l_1} \cdots \mu_{l_k}.
	\end{align*}    
    Suppose that $\fr C = \prod_i [C_i, C_i']$, recalling that $\fr C \subseteq \fr B$ forces $\max\{|C_i|, |C_i'|\} \ll X/\mu_i$ for $ 1 \le i \le s$. Thus, it follows from multidimensional partial summation that
    \begin{align*}
	    &\sum_{\bm \xi \in \fr C}  e\bigg(\partial_{\B h_1, \ldots, \B h_{j-2}, Q_{j+1}\B w} \sum_{i=j}^d \alpha_i \Psi_{\B y}^{(i)}(\bm \xi) \bigg)\\
	    &  \ll \bigg|\sum_{\bm \xi \in \fr C}  e(\partial_{\B h_1, \ldots, \B h_{j-2}, Q_{j+1}\B w} \alpha_j \Psi_{\B y}^{(j)}(\bm \xi))\bigg|
	     + \sum_{l=1}^s \frac{\mu_l}{X}\int_{C_{l}}^{C'_{l}} \bigg|\sum_{\substack{\bm \xi \in \fr C \\ \xi_l \le t}}  e(\partial_{\B h_1, \ldots, \B h_{j-2}, Q_{j+1}\B w} \alpha_j \Psi_{\B y}^{(j)}(\bm \xi))\bigg| \D t  \\
	    &\qquad+ \ldots + \frac{\mu_1 \cdots \mu_s}{X^s} \int_{\fr C} \bigg|\sum_{\substack{\bm \xi \in \fr C \\ \xi_l \le t_l\, (1 \le l \le s)}}  e(\partial_{\B h_1, \ldots, \B h_{j-2}, Q_{j+1}\B w} \alpha_j \Psi_{\B y}^{(j)}(\bm \xi))\bigg| \D \B t\\
	    & \ll \sup_{\fr C' \subseteq \fr C}\bigg|\sum_{\bm \xi \in \fr C'} e(\partial_{\B h_1, \ldots, \B h_{j-2}, Q_{j+1}\B w}\alpha_j \Psi_{\B y}^{(j)}(\bm \xi))\bigg|,
    \end{align*}
    where the supremum is over all coordinate-aligned boxes $\fr C' \subseteq \fr C$. 
    Thus, we discern that the dominant contribution arises indeed from the term of degree $j$, so that 
    \begin{align*}
        \cal W_j \ll  \dsum{\B h_l \in \fr B_{\B y}(X)}{1 \le l \le j-2} \sum_{\B w \in \fr B_{\B y}(2R_j)} \sup_{\fr C \subseteq \fr B_{\B y}(X)}  \left|\sum_{\bm \xi \in \fr C} e\bigg(\alpha_j \partial_{\B h_1, \ldots, \B h_{j-2}, Q_{j+1} \B w}  \Psi_{\B y}^{(j)}(\bm \xi) \bigg)\right|.
    \end{align*}
    
    The argument of the exponential is now linear in $\bm \xi$. 
    Since $\fr C \subseteq \fr B_{\B y}(X)$ is a box oriented along the coordinate axes, upon recalling the definition \eqref{2.7} the standard estimate on linear exponential sums yields the bound
    \begin{align*}
    	&\left|\sum_{\bm \xi \in \fr C} e\bigg(\alpha_j \partial_{\B h_1, \ldots, \B h_{j-2}, Q_{j+1} \B w}  \Psi_{\B y}^{(j)}(\bm \xi) \bigg)\right| \\
    	& \ll \prod_{m=1}^s \min\Bigg\{\frac{X}{\mu_m}, \bigg\| j!Q_{j+1}\alpha_{j} B_{\B y,m}^{(j)}(\B h_1, \ldots, \B h_{j-2}, \B w) \bigg\|^{-1} \Bigg\}. 
    \end{align*}
    Thus we have shown that $\cal W_j \ll \Upsilon_j$ and the proof of the lemma is complete. 
\end{pf}

\section{Geometry of numbers and a nonsingularity condition}

The next step is to estimate the quantity $\Upsilon_j$. 
For positive real numbers $U,V,W$ set
\begin{align}\label{3.1}
    N_{j,\B y}(U,V; W)= \card& \bigg\{ \B h_1, \ldots, \B h_{j-2} \in \fr B_{\B y}(U), \B z \in \fr B_{\B y}(V),  \nonumber \\
     & \quad \Big\|j!Q_{j+1}\alpha_{j} B_{\B y, m}^{(j)}(\B h_1, \ldots, \B h_{j-2}, \B z) \Big\| < \frac{\mu_m}{W}  \quad (1 \le m \le s)\bigg\}.
\end{align}

In this notation, standard arguments similar to those in the proof of \cite[Lemma~13.2]{dav} show that for any fixed $\theta_{j+1}, \ldots, \theta_d$ one has
\begin{align}\label{3.2}
	\Upsilon_j \ll \left(\frac{X^s}{d(\Lambda_{\B y})}\right)^{1+\eps} N_{j,\B y}(X,R_j; X).
\end{align}
Our next goal is to bound the size of $N_{j,\B y}(X,R_j; X)$. For this purpose we need a  generalisation of Davenport's lemma on the geometry of numbers (see \cite[Lemma~12.6]{dav}). 
Let $A_{k, m}>1$ be real numbers for $1 \le k \le j-1$, $1 \le m \le s$, and write 
\begin{align*}
	\scr A_k = \prod_{m=1}^s [-A_{k,m}, A_{k,m}] \qquad (1 \le k \le j-1). 
\end{align*}
Let further $0 < Z_k \le 1$ for $1 \le k \le j-1$. For any $l$ with $1 \le l \le j-1$ write $\cal R_l(Z)$ for the number of $\bm \xi_1, \ldots, \bm \xi_{j-1} \in \Z^{s}$ such that $\bm \xi_k \in Z_k \scr A_k$ for all $1 \le k \le j-1$  with $k \neq l$ and $\bm \xi_l \in Z \scr A_l$, having the property that  
\begin{align*}
	\left\|j!Q_{j+1}\alpha_{j} B_{\B y, m}^{(j)}(\bm \xi_1, \ldots,  \bm \xi_{j-1}) \right\| \le  Z  A_{l,m}^{-1} \qquad (1 \le m \le s).
\end{align*}
In this notation, Schindler and Sofos \cite{SS} give the following variant of Davenport's result. 
\begin{lem}[Lemma~2.4 in \cite{SS}]\label{L3.1}
	Fix $Z_1, \ldots, Z_{j-1} \in (0,1]$ and some $l$ with $1 \le l \le j-1$. 
	For any $Z$, $Z'$ in the range $0 < Z' \le Z \le 1$ one has 
	\begin{align*}
		\cal R_l(Z) \ll (Z / Z')^s \cal R_l(Z').  
	\end{align*}
\end{lem}

Suppose that $\theta_{j+1}, \ldots, \theta_d$ are fixed in such a way that \eqref{2.13} is satisfied. For any $\theta_j$ satisfying
\begin{align}\label{3.3}
    0< \theta_j \le 1-\omega_{j+1}-\psi (D_{d-j}+(d-1)(d-j))
\end{align}
and all $1 \le m \le s$ we set
\begin{align*}
	A_{k,m} &= (X/ \mu_m)X^{\textstyle{\frac{(1-\theta_j)(k-1)}{2}}},    &   
	Z_k &=X^{-\textstyle{\frac{(1-\theta_j)(k-1)}{2}}},    &    
	Z_k' &= X^{- \textstyle{\frac{ (1-\theta_j)(k+1)}{2}}}
\end{align*}
for $1 \le k \le j-2$, and
\begin{align*}\begin{gathered}
	A_{j-1,m}=\frac{(R_jX)^{1/2}}{\mu_i}X^{\textstyle{\frac{(1-\theta_j)(j-2)}{2}}}, \qquad 
	Z_{j-1} =\left(\frac{R_j}{X}\right)^{1/2} X^{-\textstyle{\frac{(1-\theta_j)(j-2)}{2}}},    \\
	Z_{j-1}' = \left(\frac{X}{R_j}\right)^{1/2} X^{-\textstyle{\frac{(1-\theta_j)j}{2}}}.
\end{gathered}\end{align*}
Thus $0 < Z_{k}' < Z_{k} \le 1$ for all $k$, and one has
\begin{align*}\begin{gathered}
	A_{k,m} Z_k = X/\mu_m, \qquad  A_{k,m} Z_k' = X^{\theta_j}/\mu_m, \qquad    Z_k/Z_k' = X^{1-\theta_j}, \\
	Z_k'/A_{k,m} = \mu_m X^{-1-(1-\theta_j)k} =  Z_{k+1}/A_{k+1,m}
\end{gathered}\end{align*}
for $1 \le k \le j-2$, and
\begin{align*}\begin{gathered}
	A_{j-1,m} Z_{j-1} = R_j/\mu_m, \qquad  A_{j-1,m}Z'_{j-1} = X^{\theta_j}/\mu_m, \qquad  Z_{j-1}/Z_{j-1}'= R_jX^{-\theta_j},\\
    Z_{j-1}/A_{j-1,m}=\mu_mX^{-1-(1-\theta_j)(j-2)}, \qquad Z'_{j-1}/A_{j-1,m} = \frac{\mu_m}{R_j} X^{-(j-1)(1-\theta_j)}.
\end{gathered}\end{align*}
Note here that \eqref{3.3} implies via \eqref{2.11} and \eqref{2.4} that $R_j> X^{\theta_j}$. Applying Lemma~\ref{L3.1} consecutively for the indices $k=1, \ldots, j-1$ shows that
\begin{align*}
	N_{j,\B y}(X,R_j; X) \ll X^{(j-2)(1-\theta_j)s} (R_j X^{-\theta_j})^s N_{j,\B y}(X^{\theta_j}, X^{\theta_j}; X^{(j-1)(1-\theta_j)}R_j),
\end{align*}
and hence we infer from \eqref{3.2} that
\begin{align}\label{3.4}
	\Upsilon_j \ll \frac{X^{(j-1)(1 -\theta_j)s + \eps}R_j^s}{d(\Lambda_{\B y})}N_{j,\B y}(X^{\theta_j}, X^{\theta_j};  X^{(j-1)(1-\theta_j)}R_j).
\end{align}

If we now make the assumption that $|T_{\B y}(\ba)| \gg (X^{s}/d(\Lambda_{\B y}))X^{-k_j \theta_j}$ for some $k_j>0$ and some $\theta_j$ satisfying \eqref{3.3}, we obtain from Lemma~\ref{L2.2} together with \eqref{2.6} and \eqref{3.4} the bound
\begin{align*}
	N_{j,\B y}(X^{\theta_j}, X^{\theta_j}; X^{-(j-1)(1-\theta_j)}R_j) \gg \left(\frac{X^{\theta_js}}{d(\Lambda_{\B y})}\right)^{j-1} X^{ -2^{j-1}k_j \theta_j - \eps}.
\end{align*}
The diophantine approximation condition that is implicit in \eqref{3.1} is satisfied either if the functions $B_{\B y, m}^{(j)}$ ($1 \le m \le s$) tend to vanish for geometric reasons, or if $\alpha_j$ has a good approximation in the rational numbers. Suppose that $j! B_{\B y, m}^{(j)}(\B h_1, \ldots, \B h_{j-1})$ is non-zero for some $m$ and some choice of $\B h_1, \ldots, \B h_{j-1}$ counted by $N_{j,\B y}(X^{\theta_j}, X^{\theta_j}; X^{(j-1)(1-\theta_j)}R_j)$, and denote its absolute value by $q_j$. Then $q_j \ll X^{\nu_j}|\B y|^{d-j}\mu_1$, and the approximation condition implied by the definition \eqref{3.1} takes the shape
\begin{align*}
  	\| \alpha_{j}q_jQ_{j+1}\| \ll \mu_1  X^{(j-1)(\theta_j-1)}R_j^{-1} \ll X^{-j+\omega_j}|\B y|^{D_{d-j}}\mu_1^{d-j+1}.
\end{align*}

We summarise the conclusions of our arguments in a lemma.
\begin{lem}\label{L3.2}
    Let $j \in \{2, \ldots, d\}$ be fixed. Recalling \eqref{2.8}, when $j < d$ assume that $\theta_{j+1}, \ldots, \theta_d$ are such that \eqref{2.13} is satisfied. Suppose further that for any $i$ with $j < i \le d$ there are positive integers $q_i \ll  X^{\nu_i}|\B y|^{d-i}\mu_1$ with the property that, in view of \eqref{2.9}, one has
    \begin{align*}
    	\big\|Q_i \alpha_i\big\| \ll X^{-i + \omega_i}|\B y|^{D_{d-i}}\mu_1^{d-i+1}.
    \end{align*}
    Finally, take $k_j>0$ and $\theta_j>0$ to be parameters, where $\theta_j$ satisfies \eqref{3.3}. For any $\ba \in [0,1)^{d-1}$ one of the following holds.
    \begin{enumerate}[(A)]
        \item\label{it:minor}
            The exponential sum is bounded by
            \begin{align*}
                |T_{\B y}(\ba)| \ll \frac{X^{s}}{d(\Lambda_{\B y})} X^{-k_j \theta_j+\eps}.
            \end{align*}
        \item\label{it:major}
            There exist integers $a_j$ and $q_j$ satisfying $1 \le q_j \ll X^{\nu_j}|\B y|^{d-j}\mu_1$ and $0 \le a_j \le Q_j $ such that
            \begin{align*}
                |Q_j\alpha_j - a_j| \ll X^{-j + \omega_j}|\B y|^{D_{d-j}}\mu_1^{d-j+1}.
            \end{align*}
        \item\label{it:sing}
        	The number of points $\bm \xi_1, \ldots, \bm \xi_{j-1} \in \fr B_{\B y}(X^{\theta_j})$ for which $B_{\B y,m}^{(j)}(\bm \xi_1, \ldots, \bm \xi_{j-1}) = 0$ for $1 \le m \le s$ is at least of order $(X^{\theta_js}/d(\Lambda_{\B y}))^{j-1}X^{-2^{j-1}k_j\theta_j - \eps}$.
    \end{enumerate}
\end{lem}
Our next goal is to interpret the third case geometrically. Write $\cal M_j(\B y)$ for the variety containing all $(\bm \xi_1, \ldots, \bm \xi_{j-1}) \in \A_{\C}^{(j-1)s}$ that satisfy 
\begin{align*}
	B_{\B y,m}^{(j)}(\bm \xi_1, \ldots, \bm \xi_{j-1}) = 0 \quad (1 \le m \le s).
\end{align*}
It is clear (for instance from Theorem~3.1 in \cite{browning}) that for any positive real number $Z$ one has
\begin{align*}
    \card \left\{(\B h_1, \ldots, \B h_{j-1}) \in \Z^{(j-1)s} \cap \cal M_j(\B y):\, |\B h_i| \le Z \;(1 \le i \le j-1) \right\} \ll Z^{\dim \cal M_j(\B y)}.
\end{align*}
As in the work of Schindler and Sofos \cite{SS} we cover the domain $( {\fr B}_{\B y}(X^{\theta_j}))^{j-1}$ by at most $O(\mu_1^{s(j-1)}/d(\Lambda_{\B y})^{j-1})$ translates of the box $[-X^{\theta_j}/\mu_1,X^{\theta_j}/\mu_1]^{s(j-1)}$. Suppose now that 
\begin{align}\label{3.5}
	\psi< \varpi k_j \theta_j
\end{align}
for all $j$ and a suitably small parameter $\varpi$, so that $\mu_1 \ll X^{(d-1)\varpi k_j\theta_j}$. 
Since \cite[Theorem~3.1]{browning} allows for translations, we infer that 
\begin{align*}
	&\card \left\{ (\bm \xi_1, \ldots, \bm \xi_{j-1}) \in ({\fr B}_{\B y}(X^{\theta_j}))^{j-1} \cap \cal M_j(\B y) \right\}\\ 
	&\qquad\ll \left(\frac{\mu_1^{s}}{d(\Lambda_{\B y})}\right)^{j-1} \left( \frac{X^{\theta_j}}{\mu_1} \right)^{\dim \cal M_j(\B y)}\\
	&\qquad\ll \left(\frac{X^{(d-1)\varpi k_j\theta_j s}}{d(\Lambda_{\B y})}\right)^{j-1} \left( X^{\theta_j(1-(d-1)\varpi k_j)} \right)^{\dim \cal M_j(\B y)}.
\end{align*}
We thus discern that whenever we are in case \eqref{it:sing} of Lemma~\ref{L3.2}, we must have the bound 
\begin{align*}
	\left(\frac{X^{\theta_js}}{d(\Lambda_{\B y})}\right)^{j-1} X^{ -2^{j-1}k_j\theta_j - \eps} \ll \left(\frac{X^{(d-1)\varpi k_j\theta_j s}}{d(\Lambda_{\B y})}\right)^{j-1} \left( X^{\theta_j(1-(d-1)\varpi k_j)} \right)^{\dim \cal M_j(\B y)},
\end{align*}
which simplifies to 
\begin{align*}
	(X^{\theta_j(1-(d-1)\varpi k_j )})^{(j-1)s-\dim \cal M_j(\B y)} \ll X^{2^{j-1}k_j \theta_j - \eps}.
\end{align*}
It follows that for any $j$, the case \eqref{it:sing} of Lemma~\ref{L3.2} is excluded  when
\begin{align}\label{3.6}
   	(j-1)s-\dim \cal M_j(\B y) > \frac{2^{j-1}k_j}{1-(d-1)\varpi k_j }.
\end{align}
We thus want to choose our parameters in such a way that \eqref{3.6} holds for all $2 \le j \le d$. 

We begin by observing that $\cal M_{j-1}(\B y)$ is obtained from $\cal M_j(\B y)$ by intersecting with the $s$ hyperplanes defined by $\B h_{j-1}=\B y$. This gives the inequality $\dim \cal M_{j-1}(\B y) \ge \dim \cal M_j(\B y) - s$ for all $j$ with $3 \le j \le d$, and upon solving the recursion we deduce that 
\begin{align}\label{3.7}
 	\dim \cal M_j(\B y)\le (j-2)s +  \dim \cal M_2(\B y) \qquad (2 \le j \le d).
\end{align}
It thus suffices to understand the set $\cal M_2(\B y)$. 



\begin{lem}\label{L3.3}
	Let $\B y \in \cal V$. We have $\cal M_2(\B y) = \langle \ker H_{\B y}, \B y\rangle$ and thus 
	\begin{align*}
		\dim \cal M_2(\B y) \le \dim \ker H_{\B y} +1.
	\end{align*}
\end{lem}

\begin{pf}
	It follows from the definition of $\Psi^{(2)}$ that $\cal M_2(\B y)$ is given by the set of all $\B h \in \A_{\C}^n$ satisfying $\Phi_{\B y}^{(1)}(\B h)=0$ and $(H_{\B y}\B h) \cdot \B x = 0$ for all $\B x$ having $\Phi_{\B y}^{(1)}(\B x)=0$. In particular, $\B h$ has to be such that $(H_{\B y} \B h) \cdot \B x = 0$ whenever $(H_{\B y} \B y) \cdot \B x = 0$. This is clearly satisfied if $\B h \in \ker H_{\B y}$, as then the first equation holds trivially.  On the other hand, if $\B h \not\in \ker H_{\B y}$ both equations define hyperplanes which coincide precisely if the vectors $H_{\B y}\B h$ and $H_{\B y}\B y$ are proportional, or in other words, $\B h - \alpha \B y \in \ker H_{\B y}$ for some scalar $\alpha$. Rewriting gives $\B h \in \langle \ker H_{\B y}, \B y \rangle$, and the statement follows.
\end{pf}

We now quantify the set of points $\B y$ for which $\ker H_{\B y}$ is large. For a natural number $\rho$ set 
\begin{align*}
	\cal A(\rho) = \{\B y \in \A_{\C}^n : \dim \ker H_{\B y} \le \rho-1 \}
\end{align*}
and 
\begin{align*}
	\cal B(\rho) = \{\B y \in \A_{\C}^n: \dim \ker H_{\B y} \ge \rho \},
\end{align*}
so that the sets $\cal A(\rho)$ and $\cal B(\rho)$ are complementary. Observe also that with this definition we have $\cal V^*_{2, \rho} = \cal B(\rho) \cap \cal V$ and $\cal V_{2, \rho} = \cal A(\rho) \cap \cal V$. Suppose that $\B y \in \cal A(\rho)$ for some natural number $\rho$. It then follows from \eqref{3.6} and \eqref{3.7} via Lemma~\ref{L3.3} that case \eqref{it:sing} of Lemma~\ref{L3.2} is excluded whenever the inequalities
\begin{align}\label{3.8}
	s - \rho > \frac{2^{j-1}k_j}{1-(d-1)\varpi k_j} \qquad (2 \le j \le d)
\end{align}
are satisfied. 

To conclude the section, we record the bound 
\begin{align}\label{3.9}
	 \card \{ \B y \in \cal V^*_{2,\rho}(\Z): |\B y| \le Y\} &\le \card \{ \B y \in \cal B(\rho) \cap \Z^n: |\B y| \le Y\} \ll Y^{n-\rho},
\end{align}
which follows from the argument of \cite[Lemma~2]{HB10} via Theorem~3.1 in \cite{browning}.

\section{Major and Minor arcs}\label{S4}

Lemma~\ref{L3.2} is designed to inductively define a partition into major and minor arcs for the entries $\alpha_j$ of $\ba$ as $j$ runs from $d$ to $2$. The size of the major arcs obtained in this way is controlled by the parameters $\theta_j$ and $k_j$ which it is now our job to choose optimally. Throughout this section and the next we will assume that $\B y \in \cal A(\rho)$ for some parameter $\rho$. Also, we will work on the assumption that \eqref{3.8} is satisfied, so that the singular case in Lemma~\ref{L3.2} is excluded.

Given an index~$j$ and parameters $\theta_{j}, \ldots, \theta_d$, we define the major arcs $\fr M_{\B y}(X; \theta_j, \ldots, \theta_d)$ to be the set of all $\bm \alpha \in [0,1)^{d-1}$ for which there exist integers $q_j, \ldots, q_d$ and $a_j, \ldots, a_d$  having the property that for all $i \in \{j, j+1, \ldots, d\}$ one has
\begin{equation}\label{4.1}
	\begin{gathered}
	1 \le q_i \le c_jX^{\nu_i}|\B y|^{d-i}\mu_1, \qquad 0 \le  a_i \le Q_i,   \\
	|\alpha_iQ_i-a_i| \le c_jX^{-i+\omega_i}|\B y|^{D_{d-i}} \mu_1^{d-i+1}
	\end{gathered}
\end{equation}
for some suitable constant~$c_j$. Here, we implicitly used the notation \eqref{2.8} and \eqref{2.9}. Let 
\begin{align*}
	\fr m_{\B y}(X; \theta_j, \ldots, \theta_d) =  [0,1)^{d-1} \setminus \fr M_{\B y}(X; \theta_j, \ldots, \theta_d)
\end{align*} 
be the corresponding minor arcs.
One checks that the major arcs are disjoint as soon as $X$ is sufficiently large and
\begin{align*}
    \omega_j < j/2- \psi (D_{d-j}+(d-1)(d-j+1)).
\end{align*}
The definition of the major arcs as given above is iterative in nature in that the approximation of $\alpha_j$ involves the denominators $q_i$ for all $i>j$, and this reflects the fact that our work of the previous section generates an approximation for~$\alpha_j$ only in the case when all~$\alpha_i$ with $i>j$ have already been approximated. In a sense, therefore, the major arcs $\fr M_{\B y}(X; \theta_j, \ldots, \theta_d)$ are only defined inside the set $\fr M_{\B y}(X; \theta_{j+1}, \ldots, \theta_d)$. 

At this point, we observe that the size of $T_{\B y}(\ba)$ is well defined for any particular $\ba$. In the light of Lemma~\ref{L3.2}, this means that we lose nothing by making the choice
\begin{align}\label{4.2}
    k_j \theta_j = k_i \theta_i \quad(2 \le i, j \le d).
\end{align}
With this assumption, as a consequence of our nested definition of the major arcs we have 
\begin{align*}
	|T_{\B y}(\ba)| \ll \frac{X^s}{d(\Lambda_{\B y}) }X^{-k_j \theta_j+\eps} \quad \text{ whenever $\ba \in \fr m_{\B y}(X, \theta_j, \ldots, \theta_d)$,}
\end{align*}
where the $\eps$ absorbs any possible dependence on the constants $c_j$. As the convention \eqref{4.2} renders much of the information in our above notation superfluous, we put
\begin{align*}
    \fr M_{\B y}^{(j)}(X; \theta_j)&= \fr M_{\B y}(X; \theta_j, (k_j/k_{j+1})\theta_j, \ldots, (k_j/k_{d})\theta_j),
\end{align*}
and we adopt an analogous convention for the minor arcs.

It is useful to make the definition
\begin{align*}
    \Omega_j &= \sum_{i=j}^d \omega_i = \sum_{i=j}^d (i-j+1)\nu_i \qquad (2 \le j \le d).
\end{align*}
Write further
\begin{align}\label{4.3}
	\sigma_j = \sum_{i=j}^d \frac{(i-1)}{k_i} \quad \hbox{ and } \quad \Sigma_j=\sum_{i=j}^d \sigma_j= \sum_{i=j}^d\frac{(i-j+1)(i-1)}{k_i},
\end{align}
then \eqref{4.2} implies that
\begin{align}\label{4.4}
	\omega_j = \sigma_j k_j \theta_j \quad \hbox{ and } \quad \Omega_j = \Sigma_j k_j \theta_j.
\end{align}
When there is no danger of confusion, we will employ the convention that 
\begin{align*}
	\Sigma = \Sigma_2, \qquad \sigma =\sigma_2, \qquad \omega = \omega_2.
\end{align*}
Also, define
\begin{align*}
    \Delta_j &= \sum_{i=j}^d D_{d-i} = \textstyle{\frac16}(d - j) (d - j+1) (d - j+2),
\end{align*}
noting that $\Delta_d=0$ and 
\begin{align*}
	\Delta_j \le \Delta_2 = \textstyle{\frac16} d (d^2 - 3 d + 2)\quad \text{ for all $j$}.
\end{align*} 
We then have the following simple lemma. 
\begin{lem}\label{L4.1}
  	For any $j$ with $2 \le j \le d$ the volume of the multi-dimensional major arcs is bounded by
  	\begin{align*}
    	\vol \fr M_{\B y}(X; \theta_j, \ldots, \theta_d) \ll X^{-(D-D_{j-1}) + \Omega_j + \omega_j}|\B y|^{\Delta_{j} + D_{d-j}}\mu_1^{D_{d-j+2}-1}.
     \end{align*}
\end{lem}

\begin{proof}
  	Recall the notation \eqref{2.8} and \eqref{2.9}. The condition~\eqref{4.1} implies that
  	\begin{align*}
    	&\vol \fr M_{\B y}(X; \theta_j, \ldots, \theta_d) \\
    	&\qquad  \ll  \sum_{q_j=1}^{c_jX^{\nu_j}|\B y|^{d-j}\mu_1 } \sum_{a_j = 0}^{Q_j}\left( \frac{X^{-j+\omega_j} |\B y|^{D_{d-j}}\mu_1^{d-j+1}}{Q_j} \right)\ldots \sum_{q_d=1}^{c_jX^{\nu_d}\mu_1} \sum_{a_d =  0}^{Q_d}\left( \frac{X^{-d+\omega_d}\mu_1}{Q_d} \right) \\
    	&\qquad \ll \prod_{i=j}^d X^{-i+\omega_i + \nu_i} |\B y|^{D_{d-i} + d-i}\mu_1^{d-i+2}\\
    	&\qquad \ll  X^{-(D-D_{j-1})+\Omega_j+\omega_j}|\B y|^{\Delta_j + D_{d-j}}\mu_1^{D_{d-j+2}-1}
  	\end{align*}
  	as claimed.
\end{proof}

Our next task is to analyse under which conditions the contribution of the minor arcs is under control. We first consider the one-dimensional minor arcs $\fr m_{\B y}^{(d)}(X; \theta_d)$. 

\begin{lem}\label{L4.2}
	For any choice of positive parameters $\theta_d \in (0,1]$, $k_d$ and $\delta_d$ suppose that
	\begin{align}\label{4.6}
		k_d > D-1+\delta_d
	\end{align}
	and 
	\begin{align}\label{4.7}
		(1-\Sigma_d- \sigma_d)k_d\theta_d > D_{d-1}-1 + \delta_d.
	\end{align}
	Then for some $\nu>0$ we have the bound
	\begin{align*}
		\int_{\fr m_{\B y}^{(d)}(X; \theta_d)} |T_{\B y}(\ba)| \D \ba  \ll X^{s-(D-1)- \delta_d-\nu} \mu_1.
	\end{align*}
\end{lem}

\begin{proof}
	Let $\theta_d$ be given. We can find a sequence $\theta^{(i)}_d$ with the property
	\begin{align*}
		1=\theta_d^{(0)} > \theta_d^{(1)} > \ldots > \theta_d^{(M)}=\theta_d>0
	\end{align*}
	and subject to the condition
	\begin{align}\label{4.8}
		\big(\theta_d^{(i-1)}-\theta_d^{(i)} \big) k_d < (1-\Sigma_d-\sigma_d)k_d\theta_d - (D_{d-1}-1) - \delta_d \qquad (1 \le i \le M).
	\end{align}
	Thanks to \eqref{4.7}, this is always possible with $M=O(1)$. We now infer from Lemma~\ref{L3.2} and \eqref{4.6} that
	\begin{align*}
		\int_{\fr m^{(d)}_{\B y}(X; \theta_d^{(0)})}|T_{\B y}(\ba)| \D \ba
		&\ll \sup_{\ba \in \fr m_{\B y}^{(d)}(X;\theta_d^{(0)})} |T_{\B y}(\ba)| \ll \frac{X^s}{d(\Lambda_{\B y})}X^{-k_d +\eps}\ll X^{s-(D-1)-\delta_d- \nu},
	\end{align*}
	provided that $\nu$ is small enough in terms of the other parameters.
	Further, if we write
	\begin{align*}
		\fr m_{\B y, i}^{(d)}& = \fr m^{(d)}_{\B y}(X; \theta_{d}^{(i)}) \cap \fr M^{(d)}_{\B y}(X;  \theta_{d}^{(i-1)}) \qquad (1 \le i \le M),
	\end{align*}
	one obtains via Lemma~\ref{L4.1}, \eqref{4.4}, Lemma~\ref{L3.2} and \eqref{2.4} that
	\begin{align*}
		\int_{\fr m_{\B y, i}^{(d)}}|T_{\B y}(\ba)| \D \ba & \ll \vol \fr M^{(d)}_{\B y}(X;  \theta_{d}^{(i-1)}) \sup_{\ba \in \fr m^{(d)}_{\B y}(X; \theta_{d}^{(i)})} |T_{\B y}(\ba)| \\
		& \ll \frac{X^s}{d(\Lambda_{\B y})}X^{ -(D-D_{d-1})+(\Sigma_{d}+\sigma_d)k_{d}\theta_{d}^{(i-1)}-k_d \theta_{d}^{(i)} + \eps}\mu_1^2,
	\end{align*}
	and \eqref{4.8} ensures that in the exponent one has for every $i=1, \ldots, M$ the relation
	\begin{align*}
		-k_d \theta_{d}^{(i)} + (\Sigma_{d}+\sigma_d)k_{d}\theta_{d}^{(i-1)}
		& \le \big(\theta_{d}^{(i-1)} - \theta_{d}^{(i)}\big) k_d  - \big(1 - (\Sigma_{d} + \sigma_d)\big)k_d\theta_{d}  \\
		&< -(D_{d-1}-1) - \delta_d-\nu
	\end{align*}
	for some sufficiently small $\nu>0$.
	Since $\mu_1|d(\Lambda_{\B y})$ and 
	\begin{align*}
		\int_{\fr m^{(d)}_{\B y}(X;\theta_d)} |T_{\B y}(\ba)| \D \ba = \int_{\fr m^{(d)}_{\B y}(X; \theta_d^{(0)})}|T_{\B y}(\ba)| \D \ba + \sum_{i=1}^M \int_{\fr m_{\B y, i}^{(d)}}|T_{\B y}(\ba)| \D \ba
	\end{align*}
	with $M=O(1)$, this completes the proof.
\end{proof}

We now employ an iterative argument in order to control the contribution from the nested sets of minor arcs. Fix some $j$ in the range $2 \le j \le d-1$, and suppose that the contribution arising from the sets $\fr m^{(i)}_{\B y}(X; \theta_{i})$ is already bounded for all $i>j$ and some suitable parameter $\theta^*_{j+1}$, where the $\theta_i$ with $i>j+1$ are determined by $\theta^*_{j+1}$ via \eqref{4.2}. 

\begin{lem}\label{L4.3}
    Fix an index $j$ with $2 \le j \le d-1$. Suppose that the parameters  $k_{i}$ with $j+1 \le i \le d$ as well as $\theta^*_{j+1}$ are given in accordance with \eqref{2.13}. For some $\delta_{j+1}\ge 0$ assume that  
    \begin{align}\label{4.9}
		(1-\Sigma_{j+1}- \sigma_{j+1})k_{j+1}\theta^*_{j+1} > D_{j}-1 + \delta_{j+1}.
	\end{align}    
    Furthermore, for non-negative parameters $\delta_j$ and $k_j$ suppose that $\theta_j$ satisfies \eqref{3.3} as well as the inequalities
   	\begin{align}\label{4.10}
   		0 < \theta_j < \theta_j^{(0)}= \frac{k_{j+1}}{k_j}\theta^*_{j+1}
   	\end{align}
   	and 
    \begin{align}\label{4.11}
       (1-(\Sigma_j+ \sigma_j))k_j\theta_j > D_{j-1}-1 + \delta_j.
    \end{align}
    Then the $j$-th minor arcs contribution is bounded by
    \begin{align*}
        \int_{\fr m^{(j)}(X; \theta_j)} |T_{\B y}(\ba)| \D \ba \ll  X^{s-(D-1)}\sum_{i=j}^d X^{-\delta_i - \nu}|\B y|^{\Delta_{i} + D_{d-i}} \mu_1^{D_{d-i+2}-2},
    \end{align*}
    where $\nu$ is some suitably small real number. 
\end{lem}

\begin{proof}
    Observe first that with our notation in \eqref{4.10} we have the decomposition 
    \begin{align*}
    	\fr m_{\B y}^{(j)}(X; \theta_j^{(0)}) = \fr m_{\B y}^{(j+1)}(X; \theta^*_{j+1}) \cup \left(\fr m_{\B y}^{(j)}(X; \theta_j^{(0)}) \cap \fr M_{\B y}^{(j+1)}(X; \theta^*_{j+1}) \right).
    \end{align*}
    Suppose that the lemma has been established for $j$ replaced by $j+1$, and recall \eqref{2.4} and \eqref{4.4}. 
    We infer from the inductive hypothesis and Lemmata~\ref{L4.1} and \ref{L3.2} that
    \begin{align*}
    	&\int_{\fr m^{(j)}_{\B y}(X; \theta_j^{(0)})}|T_{\B y}(\ba)| \D \ba\\
		&\ll  \int_{\fr m^{(j+1)}_{\B y}(X; \theta^*_{j+1})}|T_{\B y}(\ba)| \D \ba + \vol \fr M^{(j+1)}_{\B y}(X; \theta^*_{j+1}) \sup_{\ba \in \fr m_{\B y}^{(j)}(X;\theta_j^{(0)})} |T_{\B y}(\ba)| \\
		& \ll  \sum_{i=j+1}^d X^{s-(D-1)-\delta_i - \nu}|\B y|^{\Delta_i + D_{d-i} } \mu_1^{D_{d-i+2}-2}\\ 
		& \qquad +\frac{X^s}{d(\Lambda_{\B y})}X^{-(D-D_{j}) + (\Sigma_{j+1}+ \sigma_{j+1})k_{j+1}\theta^*_{j+1} -k_j \theta_j^{(0)}+\eps}|\B y|^{\Delta_{j+1} + D_{d-j-1} } \mu_1^{D_{d-j+1}-1}.
	\end{align*}
	Recall \eqref{4.10}. Thus the above bound implies via \eqref{4.9} and the relation $\mu_1 \ll  d(\Lambda_{\B y})$ that 
	\begin{align*}
		\int_{\fr m^{(j)}_{\B y}(X; \theta_j^{(0)})}|T_{\B y}(\ba)| \D \ba &\ll X^{s-(D-1)} \sum_{i=j+1}^d X^{-\delta_i - \nu}|\B y|^{\Delta_{i}+ D_{d-i}} \mu_1^{D_{d-i+2}-2},
    \end{align*}
    provided that $\nu$ is small enough in terms of the other parameters.

    Let now $\theta_j$ be given according to~\eqref{4.10} and \eqref{4.11}. We can find a sequence $\theta^{(i)}_{j}$ satisfying
    \begin{align*}
    	 \theta_j^{(0)} > \theta_j^{(1)} > \ldots > \theta_j^{(M)}=\theta_j>0,
    \end{align*}
    and subject to the condition
    \begin{align}\label{4.12}
    	\big(\theta_j^{(i-1)}-\theta_j^{(i)} \big) k_j < (1-(\Sigma_j+ \sigma_j))k_j\theta_j - (D_{j-1}-1) - \delta_j \qquad (1 \le i \le M). 
    \end{align}
    This is always possible with $M=O(1)$.
    For $i \ge 1$ set
    \begin{align*}
        \fr m_{\B y, i}^{(j)}&  = \fr m^{(j)}_{\B y}(X; \theta_{j}^{(i)}) \cap \fr M^{(j)}_{\B y}(X;  \theta_{j}^{(i-1)}).
    \end{align*}
    Then one deduces from Lemma~\ref{L4.1}, Lemma~\ref{L3.2}, \eqref{4.4} and \eqref{4.12} that
    \begin{align*}
        \int_{\fr m_{\B y, i}^{(j)}}|T_{\B y}(\ba)| \D \ba & \ll \vol \fr M^{(j)}_{\B y}(X;  \theta_{j}^{(i-1)}) \sup_{\ba \in \fr m_{\B y}^{(j)}(X;\theta_j^{(i)})} |T_{\B y}(\ba)|\\
        & \ll \frac{X^s}{d(\Lambda_{\B y})}X^{-(D-D_{j-1})+(\Sigma_{j}+\sigma_j)k_{j}\theta_{j}^{(i-1)}-k_j \theta_{j}^{(i)} + \eps}|\B y|^{\Delta_{j} + D_{d-j} } \mu_1^{D_{d-i+2}-1}\\
		& \ll X^{s-(D-1)-\delta_j-\nu}|\B y|^{\Delta_j + D_{d-j} } \mu_1^{D_{d-j+2}-2} 
	\end{align*}   
	for each $i \ge 1$, and thus altogether 
    \begin{align*}
        \int_{\fr m^{(j)}_{\B y}(X;\theta_j)} |T_{\B y}(\ba)| \D \ba &=  \int_{\fr m^{(j)}_{\B y}(X; \theta_j^{(0)})}|T_{\B y}(\ba)| \D \ba +  \sum_{i=1}^M \int_{\fr m_{\B y, i}^{(j)}}|T_{\B y}(\ba)| \D \ba \\
        & \ll M X^{s-(D-1)} \sum_{i=j}^d X^{-\delta_i - \nu}|\B y|^{\Delta_{i} + D_{d-i}} \mu_1^{D_{d-i+2} - 2}.
    \end{align*}
    Since $M = O(1)$, this completes the proof.
\end{proof}

We may now apply first Lemma~\ref{L4.2} and then Lemma~\ref{L4.3} successively to each of the $\theta_j$. Thus, for the initial step we need to ensure that the condition \eqref{4.6} is satisfied, and after that we have to satisfy the requirements described by \eqref{4.11} for all $2 \le j \le d$.  On the other hand, we have to be careful to ensure that in each iteration we can take $\theta_{j+1}^*$ small enough for Lemma~\ref{L3.2} to be applicable within Lemma~\ref{L4.3}. The crucial requirement here is for the bound \eqref{3.3} to be satisfied for $\theta_j^{(0)}$ for all $j$ with $2 \le j \le d-1$. Using \eqref{4.4} and our convention  \eqref{4.2}, the bound of \eqref{3.3} can be re-written in the form 
\begin{align}\label{4.13}
	(\sigma_j + 1/k_{j-1})k_j \theta_j<1  - (D_{d-j+1}+(d-1)(d-j+1))\psi \qquad (3 \le j \le d).
\end{align}
For $3 \le j \le d$, the condition \eqref{4.13} is compatible with the hypotheses \eqref{4.7} and \eqref{4.11} of Lemmata~\ref{L4.2} and \ref{L4.3}, respectively, only if
\begin{align}\label{4.14}
	\left(\sigma_j +k_{j-1}^{-1}\right)\frac{D_{j-1}-1+\delta_j}{1  - (D_{d-j+1}+(d-1)(d-j+1))\psi} + \Sigma_j + \sigma_j<  1.
\end{align}
At the same time, a comparison of \eqref{4.13} with \eqref{3.5} shows that we also require
\begin{align}\label{4.15}
	\sigma_j + \frac{1}{k_{j-1}} < \frac{(1-(D_{d-j+1}+(d-1)(d-j+1))\psi)\varpi}{\psi}.
\end{align}
Meanwhile, when $j=2$, the bound of \eqref{4.13} does not apply, and we only have the constraints stemming from \eqref{3.5} and \eqref{4.11}, which can be rewritten as  
\begin{align}\label{4.16}
	k_2\theta_2 > \max\left\{ \frac{\delta_2}{1-(\Sigma_2 +\sigma_2)}, \psi\varpi^{-1}  \right\}.
\end{align}
We will attend to this bound later, but in the meanwhile we remark that regardless of the specific values $\theta_2>0$ and $\delta_2 \ge 0$, it implies that we must have $
	\Sigma_2 + \sigma_2 < 1$.  
Summarising, we obtain the following intermediate result. 
\begin{prop}\label{P4.4}
	Assume \eqref{3.8}. Suppose that  \eqref{4.14} and \eqref{4.15} are satisfied for all $j \ge 3$, and that furthermore \eqref{4.6} and \eqref{4.16} hold. 
	Then for some $\nu > 0$ we have 
	\begin{align}\label{4.17}
		N_{\B y}(X) = \int_{\fr M_{\B y}^{(2)}(X; \theta)} T_{\B y}(\ba) \D \ba + O\left(X^{s-(D-1)-\nu}\sum_{j=2}^dX^{-\delta_j}|\B y|^{\Delta_j + D_{d-j}}\mu_1^{D_{d-j+2}-2}\right).
	\end{align}
\end{prop}
\begin{proof}
	This follows from \eqref{2.5} upon applying Lemmata \ref{L4.2} and \ref{L4.3}, and the discussion preceding the statement of the proposition. 
\end{proof}

\section{Understanding the main term}\label{S5}

In order to show that the main term of~\eqref{4.17} is indeed of the expected shape, it is necessary for the approximations of all components of $\ba$ to have the same denominator. Recall that we wrote $\omega = \omega_2$, and set
\begin{align*}
	q=Q_2\quad \text{ and }\quad b_j=(q/Q_j)a_j \qquad (2 \le j \le d).
\end{align*}
For some positive constant $c$ set $W=cX^{\omega}|\B y|^{D_{d-2} + (d-1)^2}$, where $\omega$ is as obtained in Proposition~\ref{P4.4}. Our final set $\fr P_{\B y}(X; \omega)$ of major arcs is now the set of all $\ba$ with an approximation of the shape
\begin{align}\label{5.1}
	1 \le q \le W \quad \text{and} \quad |\alpha_j  - b_j/q| \le X^{-j}W \quad (2 \le j \le d).
\end{align}
Recall \eqref{2.4}. When $c$ is sufficiently large, the set $\fr P_{\B y}(X; \omega)$ is slightly larger than $\fr M_{\B y}^{(2)} (X; \theta)$, so the corresponding minor arcs $\fr p_{\B y}(X; \omega) = [0,1)^{d-1} \setminus \fr P_{\B y}(X; \omega)$ are contained in $\fr m_{\B y}^{(2)} (X; \theta)$. In the statement of Proposition~\ref{P4.4}, we may therefore replace the major arcs $\fr M_{\B y}^{(2)} (X; \theta)$ by the larger set $\fr P_{\B y}(X; \omega)$. 

Let $\cal L_{\B y}$ denote the $s$-dimensional subspace of $\R^n$ containing $\Lambda_{\B y}$. Furthermore, we define $\cal L_{\B y}(X) = \cal L_{\B y} \cap [-X,X]^n$, and we let $\Lambda_{\B y}(q)$ denote the set of residue classes modulo $q$ of lattice points $\B x \in \Lambda_{\B y}$. Also, set 
\begin{align*}
	\vartheta_{\B y}(\bm \alpha;\B x) = \sum_{j=2}^d \alpha_j \Phi_{\B y}^{(j)} (\B x)
\end{align*}
for the analogue of $\phi_{\B y}$ in terms of the original variables $\B x \in \Lambda_{\B y}$. 
In this notation, we can now define
\begin{align}\label{5.2}
	S_{\B y}(q, \B a) = \sum_{\B x \in \Lambda_{\B y}(q)} e(\vartheta_{\B y}(\B a/q;\B x) ) \quad \text{
		and } \quad v_{\B y}(\bb, X) = \int_{\cal L_{\B y}(X)} e(\vartheta_{\B y}(\bb; \bm \xi)) \D \bm{\xi}.
\end{align}
These functions allow us to approximate the exponential sum $T_{\B y}(\ba)$ on the major arcs. 

\begin{lem}\label{L5.1}
	Suppose that $\ba = \B a /q + \bb$ with $q \le X^{1-\psi(d-1)}$. We have 
	\begin{align}\label{5.3}
		\left|T_{\B y}(\ba) -  \frac{S_{\B y}(q, \B a)}{q^s} \frac{v_{\B y}(\bb, X)}{d (\Lambda_{\B y})}\right| 
		&\ll X^{s-1}q\left( 1 + \frac{1}{d(\Lambda_{\B y})}\sum_{j=2}^d |\beta_j|X^{j}|\B y|^{d-j} \right).
	\end{align}
\end{lem}
\begin{proof}
	This is essentially standard, but due to our specific setting over a lattice we prefer to provide a full proof. 	
	Sorting the terms into arithmetic progressions modulo $q$, we find that 
	\begin{align*}
		T_{\B y}(\ba) = \sum_{\B z \in \Lambda_{\B y}(q)} e (\vartheta_{\B y}(\B a/q;\B z)) \sum_{\substack{\B w \in \Lambda_{\B y} \\			
		q \B w + \B z \in \fr A_{\B y}(X)}} e (\vartheta_{\B y}(\bb;q \B w + \B z)),
	\end{align*}
	and hence 
	\begin{align*}
		\left|T_{\B y}(\ba) -  \frac{S_{\B y}(q, \B a)}{q^s} \frac{v_{\B y}(\bb, X)}{d (\Lambda_{\B y})}\right| 
		& \ll \sum_{\B z \in \Lambda_{\B y}(q)} e (\vartheta_{\B y}(\B a/q;\B z)) H(q, \B z, \bb),
	\end{align*}
	where 
	\begin{align*}
		H(q, \B z, \bb) = \sum_{\substack{\B w \in \Lambda_{\B y} \\q \B w + \B z \in \fr A_{\B y}(X)}} e (\vartheta_{\B y}(\bb; q \B w + \B z)) - \frac{1}{q^s d(\Lambda_{\B y})}  \int_{\bm{\xi} \in \cal L_{\B y}(X)} e(\vartheta_{\B y}(\bb;\bm \xi)) \D \bm{\xi}.
	\end{align*}
	
	Denote the fundamental domain of $\Lambda_{\B y}$ by $\cal F$, and for $\B w \in \Lambda_{\B y}$ write $\cal F(\B w) = \B w + \cal F$ for the fundamental domain located at $\B w$. Moreover, we write $\cal F_{q, \B z}(\B w) = q (\B w + \cal F)+\B z $ for the domain, stretched by a factor $q$, that is located at $q \B w + \B z$. We want to replace $H(q, \B z, \bb)$ by the related quantity 
	\begin{align*}
		H^*(q, \B z, \bb) =\sum_{\substack{\B w \in \Lambda_{\B y} \\q \B w + \B z \in \fr A_{\B y}(X)}} \left\{e (\vartheta_{\B y}(\bb; q \B w + \B z)) - \frac{1}{q^s d(\Lambda_{\B y})}  \int_{ \cal F_{q, \B z}(\B w)} e(\vartheta_{\B y}(\ba;\bm \xi)) \D \bm{\xi}\right\}.
	\end{align*}
	Clearly, we have  $\vol \cal F = d(\Lambda_{\B y})$ and $\vol \cal F_{q, \B z}(\B w) = q^s d(\Lambda_{\B y})$. Thus, $\cal L_{\B y}(X)$ may be covered by $O(X^s/(q^s d(\Lambda_{\B y})))$  domains $\cal F_{q, \B z}(\B w)$ as $\B w$ varies over $\Lambda_{\B y}$, and the boundary intersects $\ll (X/q)^{s-1}\mu_1/d(\Lambda_{\B y}) \ll (X/q)^{s-1}$ of these. Thus, the defect is of size at most $O(X^{s-1}q d(\Lambda_{\B y}) )$. With this information, we find upon partitioning the integrating domain that $H(q, \B z, \bb) - H^*(q, \B z, \bb)\ll (X/q)^{s-1}$, and thus 
	\begin{align*}
		\left|T_{\B y}(\ba) -  \frac{S_{\B y}(q, \B a)}{q^s} \frac{v_{\B y}(\bb, X)}{d (\Lambda_{\B y})}\right| 
		& \ll \sum_{\B z \in \Lambda_{\B y}(q)} e (\vartheta_{\B y}(\B a/q;\B z)) H^*(q, \B z, \bb) + O(X^{s-1}q).
	\end{align*}
	Rewriting
	\begin{align*}
		H^*(q, \B z, \bb)
		=  \sum_{\substack{\B w \in \Lambda_{\B y} \\q \B w + \B z \in \fr A_{\B y}(X)}}\frac{1}{d(\Lambda_{\B y})}\int_{\cal F(\B w)}e (\vartheta_{\B y}(\bb; q \B w + \B z)) - e(\vartheta_{\B y}(\bb; q \bm \xi + \B z)) \D \bm{\xi}
	\end{align*}
	puts us into a position where we can apply the mean value theorem, whereupon we see that
	\begin{align*}	
		H^*(q, \B z, \bb)
		&\ll \sum_{\substack{\B w \in \Lambda_{\B y} \\q \B w + \B z \in \fr A_{\B y}(X)}}q \sum_{j=2}^d |\beta_j|X^{j-1}|\B y|^{d-j}
		\ll\left(\frac{X^s}{q^sd(\Lambda_{\B y})}+1\right)q \sum_{j=2}^d |\beta_j|X^{j-1}|\B y|^{d-j}.
	\end{align*}	
	The desired bound follows now upon applying the trivial bound $S_{\B y}(q,\B a) \ll q^s$. 
\end{proof}

In particular, when $\ba \in \fr P_{\B y}(X;\omega)$, inserting the conditions \eqref{5.1} into \eqref{5.3} shows that
\begin{align*}
	\left|T_{\B y}(\ba) -  \frac{S_{\B y}(q, \B a)}{q^s} \frac{v_{\B y}(\bb, X)}{d (\Lambda_{\B y})}\right| \ll  X^{s-1}W^2.
\end{align*}
Since
\begin{align*}
	\vol \fr P_{\B y}(X;\omega) &\ll \sum_{q=1}^{W}\prod_{j=2}^d q X^{-j}W	\ll X^{-(D-1)}W^{2d-1},
\end{align*}
it follows that
\begin{align}\label{5.4}
	\int_{\fr P_{\B y}(X; \omega)} T_{\B y}(\ba) \D \ba& = \sum_{q=1}^{W}  \dsum{\B a = 0}{(\B a, q)=1}^{q-1} \frac{S_{\B y}(q, \B a)}{q^s} \int_{\substack{|\beta_j| \le  X^{-j}W \\ (2 \le j \le d)}} \frac{v_{\B y}(\bb, X)}{d(\Lambda_{\B y})} \D \bb + O \left( X^{s-D}W^{2d+1} \right).
\end{align}

As usual, the growth rate of the main term in the asymptotic formula comes from the contribution of $v_{\B y}(\bb, X)$. 
Setting $\gamma_j = X^{j}\beta_j$ for $2 \le j \le d$, the identity
\begin{align}\label{5.5}
	v_{\B y}(\bb, X) =X^{s} v_{\B y}(\bg, 1),
\end{align}
follows from \eqref{5.2} by applying integration by parts, and in the same manner one finds further that
\begin{align*}
  	\int_{\substack{|\beta_j| \le X^{-j}W  \\ (2 \le j \le d)} }v_{\B y}(\bb, X) \D \bb = X^{s-(D-1)} \int_{|\bb| \le W } v_{\B y}(\bb, 1) \D \bb.
\end{align*}
Let
\begin{align*}
    \fr J_{\B y}(W) = \int_{[-W,W]^{d-1}} \frac{v_{\B y}(\bb, 1)}{d(\Lambda_{\B y})} \D \bb \qquad \text{ and }\qquad  \fr S_{\B y}(W) = \sum_{q=1}^{W} q^{-s} \sum_{\substack{\B a = 0 \\ (\B a, q)=1}}^{q-1} S_{\B y}(q, \B a),
\end{align*}
then we can rewrite \eqref{5.4} in the shape 
\begin{align}\label{5.6}
  	\int_{\fr P_{\B y}(X; \omega)} T_{\B y}(\ba) \D \ba &= X^{s-D+1} \fr S_{\B y}(W) \fr J_{\B y}(W) + O\left( X^{s-D}W^{2d+1}\right).
\end{align}


In order to understand the main term in~\eqref{5.6}, we extend the  truncated singular integral $\fr J_{\B y}(W)$ and the truncated singular series $\fr S_{\B y}(W)$ to infinity by taking the limits $X \rightarrow \infty$ in both expressions. In our analysis of these limits, the notations $\bb_j = (\beta_j, \ldots, \beta_d)$ and $\B a_j = (a_j, \ldots, a_d)$ ($2 \le j \le d$) will prove useful. 

We start by considering the singular integral.
\begin{lem}\label{L5.2}
    We have
    \begin{align*}
        |v_{\B y}(\bb, 1)| \ll \min_{2 \le j \le d}  |\B y|^{D_{d-j}/\sigma_j}\mu_1^{(d-j+1)/\sigma_j}(1+ |\bb_j|)^{-1/\sigma_j+\eps}.
    \end{align*}
\end{lem}

\begin{proof}
    Fix $j$ with $2 \le j \le d$. For $|\bb_j| \le 1$ the claim is trivial, so we may assume that $|\bb_j| > 1$. Choose $P=|\bb|^A$ for some large parameter $A$ to be fixed later, and write $\bg = (P^{-2} \beta_2, \ldots, P^{-d} \beta_d)$ and $\bg_j = (\gamma_j, \ldots, \gamma_d)$. 
    Recalling \eqref{4.4}, we fix $\theta_j$ such that
    \begin{align*}
	   \max_{j \le i \le d } \frac{|\bb_i|}{c_j P^{\omega_i}  |\B y|^{D_{d-i}}\mu_1^{d-i+1}} = 1,
    \end{align*}
    so that 
    \begin{align}\label{5.7}
    	P^{ -k_j\theta_j} \ll |\bb_j|^{-1/\sigma_j} |\B y|^{D_{d-j}/\sigma_j}\mu_1^{(d-j+1)/\sigma_j}.
    \end{align}
    With this choice, $\bm \gamma_j$ lies in the major arcs $\fr M_{\B y}^{(j)}(P; \theta_j)$. Clearly, the major arcs are disjoint when $A$ is sufficiently large, so $\bg_j$ is best approximated by $q=1$ and $\B a_j= \bm 0$.  We therefore have from Lemma \ref{L5.1} and \eqref{5.5} that
    \begin{align}\label{5.8}
    	|v_{\B y}(\bb, 1)| \ll \left(\frac{P^s}{d(\Lambda_{\B y})}\right)^{-1} |T_{\B y}(\bg; P)| + P^{-1}|\bb|.
    \end{align} 
    
    On the other hand, $\bm \gamma$ lies just on the boundary of the major arcs and thus by continuity the minor arcs bound continues to apply. Consequently, we obtain from Lemma~\ref{L3.2} and \eqref{5.7} the complementary estimate 
    \begin{align*}
    	|T_{\B y}(\bg; P)| \ll \left( \frac{P^s}{d(\Lambda_{\B y})}\right)P^{-k_j \theta_j + \eps} \ll \left( \frac{P^s}{d(\Lambda_{\B y})}\right) P^{\eps}|\bb_j|^{-1/\sigma_j}  |\B y|^{D_{d-j}/\sigma_j}\mu_1^{(d-j+1)/\sigma_j}.
    \end{align*}
    Inserting this into \eqref{5.8} leads to 
    \begin{align*}
	   |v_{\B y}(\bb, 1)| \ll P^{\eps}|\bb_j|^{-1/\sigma_j+\eps}   |\B y|^{D_{d-j}/\sigma_j}\mu_1^{(d-j+1)/\sigma_j} + P^{-1}|\bb|,
    \end{align*}
   	and upon recalling that $P=|\bb|^A \ge |\bb_j|^A$, this reproduces the desired estimate whenever $A$ is sufficiently large.
\end{proof}

It follows from Lemma~\ref{L5.2} that for any tuple $\lambda_{2}, \ldots, \lambda_{d} \in [0,1]$ satisfying $\lambda_{2} + \ldots + \lambda_{d}=1$ we have
\begin{align*}
	\int_{\substack{\substack{\bb \in \R^{d-1} \\ |\bb| > W}}}\frac{|v_{\B y}(\bb, 1)|}{d(\Lambda_{\B y})} \D \bb
	&\ll \frac{1}{d(\Lambda_{\B y})} \int_{\substack{\substack{\bb \in \R^{d-1} \\ |\bb| > W}}} \prod_{j=2}^d\left( |\B y|^{D_{d-j}/\sigma_j}\mu_1^{(d-j+1)/\sigma_j} (1+ |\bb|)^{- 1/\sigma_j+\eps}\right)^{\lambda_{j}}\D \bb.
\end{align*} 
The set of all $\bb \in \R^{d-1}$ having $|\bb|=r$ has volume $O(r^{d-2})$. Recalling that $\mu_1 \ll d(\Lambda_{\B y})$, it follows that the above integral is bounded by   
\begin{align*}
	\int_{\substack{\substack{\bb \in \R^{d-1} \\ |\bb| > W}}}\frac{|v_{\B y}(\bb, 1)|}{d(\Lambda_{\B y})} \D \bb
	&\ll  |\B y|^{\kappa_1 }\mu_1^{-1+\kappa_2}  \int_{r > W}  (1+ r) ^{-\kappa_3+d-2+ \eps}\D r,
\end{align*} 
where
\begin{align*}
	\kappa_1 = \sum_{j=2}^d \frac{D_{d-j}\lambda_j}{\sigma_j}, \qquad \kappa_2 = \sum_{j=2}^d \frac{(d-j+1)\lambda_j}{\sigma_j}, \qquad\kappa_3  =\sum_{j=2}^d \frac{\lambda_j}{\sigma_j} .
\end{align*}
The integral converges if we can pick $\lambda_2, \ldots, \lambda_d$ in such a way that $\kappa_3> d-1$.
We take $\lambda_j = \sigma_j$ for $j \ge 3$, so that $\lambda_2 = 1 - \Sigma_3=\sigma_2 + (1-\Sigma_2)$. With this choice, the desired inequality $\kappa_3 > d-1$ is satisfied if $\Sigma  < 1$, and we have
\begin{align*}
	\kappa_3 = d-1 + \frac{1-\Sigma_2}{\sigma_2} = d+\frac{1-\Sigma-\sigma}{\sigma}.
\end{align*}
Moreover, using these values in our expression for $\kappa_1$ and $\kappa_2$ we obtain
\begin{align*}
	\kappa_1 =  \Delta_2 + D_{d-2} + D_{d-2}\frac{1-\sigma - \Sigma}{\sigma}, \qquad \kappa_2 =  D-1+ (d-1)\frac{1-\sigma - \Sigma}{\sigma}.
\end{align*}
Upon referring to \eqref{2.4}, this allows us to conclude that
\begin{align}\label{5.9}
    \fr J_{\B y} - \fr J_{\B y}(W) & \ll|\B y|^{\Delta_2 + D_{d-2}+ (d-1)(D-2)+ (D_{d-2}+(d-1)^2){\textstyle \frac{1-\sigma -\Sigma}{\sigma}}}  W^{-1-{\textstyle \frac{1-\sigma -\Sigma}{\sigma}}+\eps} \nonumber \\
    &\ll |\B y|^{\frac13(2 d^3 - 11 d + 9) + \frac12(3d^2-7d+4){\textstyle \frac{1-\sigma -\Sigma}{\sigma}}} W^{-1-{\textstyle \frac{1-\sigma -\Sigma}{\sigma}}+\eps},
\end{align}
and we have the bound 
\begin{align}\label{5.10}
	\fr J_{\B y}(W) & \ll |\B y|^{\frac13(2 d^3 - 11 d + 9) + \frac12(3d^2-7d+4){\textstyle \frac{1-\sigma -\Sigma}{\sigma}}}
\end{align}
uniformly in $W$.

The next step is to complete the truncated singular series. 
\begin{lem}\label{L5.3}
    The terms of the singular series are bounded by
    \begin{align*}
		|q^{-s}S_{\B y}(q, \B a) | \ll \min_{2 \le j \le d} q^{\eps}\left(\frac{q}{(q, \B a_j)} \right)^{-1/\sigma_j} |\B y|^{D_{d-j}/\sigma_j}\mu_1^{(d-j+1)/\sigma_j}.
    \end{align*}
\end{lem}

\begin{proof}
    For $q=1$ the estimate is trivial, so we may suppose that $q > 1$. Fix $P = q^A$ for some large $A$ to be determined later. For any $j$ with $2 \le j \le d$ fix $\theta_j$ such that
    \begin{align*}
    	\max_{j \le i \le d}\frac{q/(q, \B a_i)}{ c_j^{d-i} P^{\omega_i}|\B y|^{D_{d-i}} \mu_1^{d-i+1}}=1,
    \end{align*}
    so that in particular
    \begin{align}\label{5.11}
     	P^{-k_j \theta_j}	 \ll\left( \frac{q}{(q,\B a_j)}\right)^{-1/\sigma_j}|\B y|^{D_{d-j}/\sigma_j}\mu_1^{(d-j+1)/\sigma_j}
    \end{align}
    and $\B a_j/q \in \fr M_{\B y}^{(j)}(P; \theta_j)$. Note that by taking $A$ sufficiently large we may ensure that the major arcs $\fr M_{\B y}^{(j)}(P; \theta_j)$ are disjoint, so $\B a_j/q$ is best approximated by itself. Applying Lemma~\ref{5.1} and \eqref{5.5} with $\bb = \bm 0$ and observing that $v_{\B y}(\bm 0, 1) \asymp 1$, it follows that
    \begin{align}\label{5.12}
    		q^{-s}S_{\B y}(q, \B a)  \ll \left(\frac{P^s}{d(\Lambda_{\B y})}\right)^{-1} |T_{\B y}(q^{-1} \B a; P)| + P^{-1}q .
    \end{align}
    At the same time, $\B a_j/q$ can be viewed as lying just on the boundary of the major arcs in the $q$-aspect. As before, this implies that Lemma \ref{L3.2} and \eqref{5.11} furnish the additional minor arcs bound 
    \begin{align*}
    	|T_{\B y}(q^{-1} \B a; P)|\ll d(\Lambda_{\B y})^{-1}P^{s-k_j \theta_j+\eps} \ll\frac{P^{s+\eps}}{d(\Lambda_{\B y})} \left(\frac{q}{(q, \B a_j)}\right)^{-1/\sigma_j}|\B y|^{D_{d-j}/\sigma_j}\mu_1^{(d-j+1)/\sigma_j},
    \end{align*}
   and on substituting this into \eqref{5.12} we discern that
    \begin{align*}
    		q^{-s}S_{\B y}(q, \B a) \ll P^{\eps}\left(\frac{q}{(q, \B a_j)}\right)^{-1/\sigma_j}|\B y|^{D_{d-j}/\sigma_j}\mu_1^{(d-j+1)/\sigma_j} + P^{-1}q \qquad (2 \le j \le d).
    \end{align*}
    Recalling that $P=q^A$, it is clear that for $A$ sufficiently large the first term dominates.
\end{proof}

Lemma~\ref{L5.3} implies that the singular series may be extended to infinity.
Let $\tau_2, \ldots, \tau_d$ be natural numbers with the property that $\tau_j | \tau_{j+1}$ for $2 \le j \le d-1$ and $\tau_d|q$. For any $j$ the number of choices of $\B a \mmod q$ satisfying $(q, \B a_j)= \tau_j$ is $O(q^{d-1}/\tau_j^{d-j+1})$. It thus follows that we have
\begin{align*}
	\sum_{q=1}^W \sum_{\substack{\B a= 0 \\ (\B a, q)=1}}^{q-1} q^{-s} |S_{\B y}(q, \B a)|
    & \ll \sum_{q=1}^W  \sum_{\tau_2| \ldots| \tau_d | q}   \min_{2 \le j \le d} q^{j-2+\eps} \left(\frac{q}{\tau_j}\right)^{d-j+1-1/\sigma_j}  |\B y|^{D_{d-j}/\sigma_j}\mu_1^{(d-j+1)/\sigma_j}\\
    &\ll \sum_{q=1}^W q^{d-1+\eps} \prod_{j=2}^d \left(q^{-1/\sigma_j} |\B y|^{D_{d-j}/\sigma_j}\mu_1^{(d-j+1)/\sigma_j}\right)^{\lambda_j}
\end{align*}
for any choice of $\lambda_2, \ldots, \lambda_d \in [0,1]$ with $\lambda_2+ \ldots + \lambda_d=1$. Just like in the treatment of the singular integral, we can take $\lambda_j = \sigma_j$ for $3 \le j \le d$, and $\lambda_2 = 1 - \Sigma_3$. 
This choice yields the bound
\begin{align}\label{5.13}
	\fr S_{\B y} - \fr S_{\B y}(W) &\ll |\B y|^{ \Delta_2 + D_{d-2}+ (d-1)(D-1)+ {\textstyle \frac{1-\sigma-\Sigma}{\sigma}}(D_{d-2}+(d-1)^2) } \sum_{q\ge W} q^{-1-{\textstyle \frac{1-\sigma-\Sigma}{\sigma}}+\eps} \nonumber\\
	&\ll |\B y|^{\frac23(d^3 - 4 d + 3) + \frac12(3d^2-7d+4){\textstyle \frac{1-\sigma -\Sigma}{\sigma}}} W^{-{\textstyle \frac{1-\sigma -\Sigma}{\sigma}}+\eps}
\end{align}
whenever we have $\Sigma+\sigma< 1$. Again, we recall that this last inequality is satisfied as a consequence of the more stringent condition \eqref{4.16}. In particular, we have the bound 
\begin{align}\label{5.14}
	\fr S_{\B y}(W) \ll |\B y|^{\frac23(d^3 - 4 d + 3) + \frac12(3d^2-7d+4){\textstyle \frac{1-\sigma -\Sigma}{\sigma}}},
\end{align}
which holds uniformly in $W$. 

We can now complete the singular series and integral. Here, from \eqref{5.9}, \eqref{5.10},  \eqref{5.13} and \eqref{5.14} and upon inserting our value $W=c X^{\omega}|\B y|^{D_{d-2}+(d-1)^2}$, we find that 
\begin{align}\label{5.15}
	|\fr J_{\B y} \fr S_{\B y} - \fr J_{\B y}(W) \fr S_{\B y}(W) |
	&\ll |\B y|^{\frac13(4 d^3 - 19 d + 15) + (3d^2-7d+4){\textstyle \frac{1-\sigma -\Sigma}{\sigma}}} W^{- {\textstyle \frac{1-\sigma -\Sigma}{\sigma}}+\eps}\nonumber\\
	&\ll X^{-\omega{\textstyle \frac{1-\sigma -\Sigma}{\sigma}}+\eps}|\B y|^{\frac13(4 d^3 - 19 d + 15) + \frac12(3 d^2 - 7 d + 4)        {\textstyle \frac{1-\sigma -\Sigma}{\sigma}}}.
\end{align}

It remains to collect our estimates. 
\begin{prop}\label{P5.4}
	Make the assumption \eqref{2.10} and suppose that the conditions \eqref{4.6}, \eqref{4.14}, \eqref{4.15} and \eqref{4.16} are satisfied. Moreover, assume \eqref{3.8}. In this case we have the asymptotic formula 
	\begin{align*}
		N_{\B y}(X)= X^{n-D} \left(\fr S_{\B y} \fr J_{\B y} + O(E(\B y, \theta))\right),
	\end{align*}
	where
	\begin{align}\label{5.16}
		E(\B y, \theta) &= \sum_{j=2}^dX^{-\delta_j-\nu}|\B y|^{\Delta_j+D_{d-j}+ (D_{d-j+1}+d-j)(d-1)} + X^{-1+(2d+1)\omega} |\B y|^{\frac12(6d^3-11d^2+d+4)} \nonumber\\
		& \qquad + X^{-\omega{\textstyle \frac{1-\sigma -\Sigma}{\sigma}}+\eps}|\B y|^{\frac13(4 d^3 - 19 d + 15) + \frac12(3 d^2 - 7 d + 4) {\textstyle \frac{1-\sigma -\Sigma}{\sigma}}}.
	\end{align}
\end{prop}
\begin{proof}
	Recall that we had $n=s+1$. The statement now follows from Proposition~\ref{P4.4} together with \eqref{5.6} and \eqref{5.15}.
\end{proof}
Before concluding the section, we remark that the singular series and integral can be expressed in terms of solution densities of the system \eqref{2.2} over the real and $p$-adic numbers.  
Indeed, since under the hypotheses of the proposition the singular series is absolutely convergent, by standard arguments it can be written as an absolutely convergent Euler product $\fr S_{\B y} = \prod_{p} \chi_p$, where
\begin{align*}
	\chi_p &= \sum_{h=0}^{\infty} p^{-hs}\sum_{\substack{\B a = 1 \\(\B a, p)=1}}^{p^h} S_{\B y}(p^h, \B a) \\
	& = \lim_{H \to \infty}p^{H(D-1-s)} \#\{\B x \in \Lambda_{\B y}(p^H): \Phi^{(j)}_{\B y} (\B x) \equiv 0 \mmod{p^H} \text{ for }2 \le j \le d\}. 
\end{align*}
Upon recalling that $\Lambda_{\B y}(q)$ denotes the set of all $\B x \in (\Z/q\Z)^n$ that satisfy the congruence $\Phi_{\B y}^{(1)}(\B x) \equiv 0 \mmod q$, we see that the above can be re-written as 
\begin{align*}
	\chi_p &= \lim_{H \to \infty}p^{H(D-n)} \#\{\B x \in (\Z / p^H \Z)^n: \Phi^{(j)}_{\B y} (\B x) \equiv 0 \mmod{p^H} \text{ for }1 \le j \le d\}. 
\end{align*}
Thus, each factor $\chi_p$ reflects the solution density of \eqref{2.2} in $\Q_p$. 

For the singular integral we proceed in a similar manner. Recall that $\Phi_{\B y}^{(1)}$ is an invertible linear transformation. Consider the manifold $M(h) = \{ \bm \xi \in [-1,1]^n: \Phi_{\B y}^{(1)}(\bm \xi) = h\}$ with associated measure $\mu$, normalised such that $\mu(M(0))=d(\Lambda_{\B y})^{-1}$. Let now
\begin{align*}
	g(\bm \xi) = \int_{\R^{d-1}} e\left(\sum_{j=2}^d \eta_j \Phi^{(j)}_{\B y}(\bm \xi)\right) \D \bm \eta \qquad \text{and} \qquad
	f(h) = \int_{M(h)}  g(\bm \xi)  \D \mu(\bm \xi),
\end{align*}
so that $f(0) = \fr J_{\B y}$. 
The inverse Fourier transform of $f$ is given by 
\begin{align*}
	\cal F^{-1} f(\alpha)&=  \int_{[-1,1]^n} g(\bm \xi) e(\alpha \Phi_{\B y}^{(1)}(\bm \xi))\D \bm \xi,
\end{align*}
and upon taking the (regular) Fourier transform it follows from the Fourier inversion formula that 
\begin{align*}
	f(N) &= \int_{\R} \int_{[-1,1]^n} g(\bm \xi) e(\alpha (\Phi_{\B y}^{(1)}(\bm \xi) - N))\D \bm \xi \D \alpha. 
\end{align*}
Thus we conclude that 
\begin{align*}
	\fr J_{\B y} = f(0) = \int_{[-1,1]^n}\int_{\R^{d}} e\left(\sum_{j=1}^d \eta_j \Phi^{(j)}_{\B y}(\bm \xi)\right) \D \bm \eta \D \bm \xi.
\end{align*}
One can now show by standard arguments (for instance Lemma~2 and \S 11 in \cite{wms82-quad}) that this expression indeed describes the solution density of \eqref{2.2} over the real unit hypercube.

\section{Endgame}\label{S6}

The quantities $\sigma_j$ and $\Sigma_j$ can be expressed in terms of $s$ itself. It is a straightforward exercise to confirm the identities  
\begin{align}\label{6.1}
	\sum_{n=1}^N n2^n = 2^{N+1}(N-1)+2 \qquad \text{ and } \qquad \sum_{n=1}^N n^2 2^n = 2^{N+1}(N^2-2N+3)-6.
\end{align} 
Note that \eqref{3.8} transforms into 
\begin{align*}
	\frac{1}{k_j} > \frac{2^{j-1}}{s-\rho} + (d-1)\varpi.
\end{align*}
Using this within \eqref{4.3}, an application of \eqref{6.1} produces the bounds
\begin{align*}
	\sigma_j &> 
	 \frac{2^d(d-2)-2^{j-1}(j-3) }{s-\rho} + \varpi(d-1)\frac{d(d-1)- (j-1)(j-2)}{2}
\end{align*}
and 
\begin{align*}
	\Sigma_j &> 
	\frac{2^d(d^2-2d+2-j(d-2))+2^{j-1}(j-5)}{s - \rho}\\
	& \qquad+ \varpi \frac{(d-1) (d - j+1) (d - j+2) (2 d + j-3)}{6},
\end{align*}
which we require to hold for all indices $j$ in our range $2 \le j \le d$. For the sake of simplicity we replace all these bounds by 
\begin{align}\label{6.2}
	\begin{gathered}
	\frac{1}{k_j} > \frac{2^{d-1}}{s-\rho} + (d-1)\varpi, \qquad\sigma_j > \frac{2^d (d-1)}{s-\rho} + \frac{\varpi d(d-1)^2}{2},  \\
	 \Sigma_j > \frac{2^d (d^2-4d+6)}{s-\rho}+ \frac{\varpi (2d-1)d(d-1)^2}{6}.
	 \end{gathered}
\end{align}
This allows us to state a first result.
\begin{thm}\label{T6.1}
	Let $F \in \Z[x_1, \ldots, x_n]$ be a non-singular form of degree $d \ge 5$ defining a hypersurface $\cal V$. Let further $\psi>0$ be a parameter satisfying 
	\begin{align}\label{6.3}
		\psi^{-1}> {\textstyle d^4 + \frac32 d^3 - \frac{11}{2} d^2 + d + 2},  
	\end{align} 
	and set
	\begin{align*}
		n_1(\psi)&=  \frac{2^{d-1}\left(d^3 + \frac12d^2 - \frac{11}{2} d + 10 - \psi p_6(d)\right)}{1-(d^4 + \frac32 d^3 - \frac{11}{2} d^2 + d + 2)\psi}, 
	\end{align*}	
	where $p_6(d)=\frac{1}{12}(50 d^6 - 171 d^5 + 88 d^4 + 517 d^3 - 732 d^2 + 8 d - 120)$.
	For some integer $\rho \in [1,n]$ suppose that $n-\rho > n_1(\psi)$. Then there exists a real positive number $\nu$ with the property that 
	\begin{align*}
		N_{\B y}(X) = X^{n-\frac12 d(d+1)}\fr S_{\B y} \fr J_{\B y} + O(X^{n-\frac12 d(d+1) - \nu}) 
	\end{align*}
	uniformly for all $\B y \in \cal V_{2,\rho}(\Z)$ satisfying $|\B y| \le X^{\psi}$, and the factors satisfy $0 \le \fr S_{\B y} \ll_{\B y} 1$ and $0 \le \fr J_{\B y} \ll_{\B y} 1$.
\end{thm}

\begin{pf}
	Our main task here is to bound the error terms given by \eqref{5.16} in the conclusion of Proposition \ref{P5.4}, while at the same time ensuring that the hypotheses of said proposition are satisfied. In order to control the first term in \eqref{5.16} we choose 
	\begin{align*}
		\delta_j = \psi (\Delta_j+D_{d-j}+ (D_{d-j+2} - 2)(d-1))
	\end{align*} 
	for $2 \le j \le d$. Thus, we have $\delta_d=(d-1)\psi$. 	
	With this choice, and recalling \eqref{2.10}, the bound in \eqref{4.6} is certainly majorised by $k_d > D$. In a similar manner, upon taking into account the uniform bounds \eqref{6.2} as well as the relations $D_j \le D$ and 
	\begin{align*}
		\delta_j \le \delta_2=\textstyle{\frac13}(2d^3 - 11d + 9)\psi
	\end{align*} 
	for all $j$, a modicum of computation reveals that for all $\psi$ satisfying \eqref{6.3} one has 
	\begin{align*}
		\frac{D_{j-1}-1+\delta_j}{1  - (D_{d-j+1}+(d-1)(d-j+1))\psi} \le \frac{d(d-1)}{2},
	\end{align*}
	and hence the condition \eqref{4.14} may be simplified to
	\begin{align}\label{6.4}
		\textstyle{\frac12 d(d-1)}\left(\sigma_j +k_{j-1}^{-1}\right) + \Sigma_j + \sigma_j<  1.
	\end{align}
	Upon inserting \eqref{6.2}, we see that the conditions \eqref{4.6}, \eqref{4.14} (as simplified to \eqref{6.4}) and \eqref{4.15} of Proposition~\ref{P5.4} are satisfied whenever $s - \rho > \max\{a_0(\varpi) , a_1(\varpi) , a_2(\varpi,\psi) \}$, where
	\begin{align*}
		a_0(\varpi) &= \frac{2^{d-2} d(d+1)}{1-\frac12 d(d^2-1)\varpi},\\
		a_1(\varpi)  &= \frac{2^{d-2}(2 d^3 + d^2 - 11 d + 20)}{1 - \frac{1}{12} (d - 1)^2 d (3 d^2 + d + 10)\varpi},\\
		a_2(\varpi,\psi)  &= \frac{2^{d-1}(2d-1)\psi}{\varpi(1- \frac{1}2d(d^2+d-2)\psi)}.
	\end{align*}
	For this to be defined, we require in particular that 
	\begin{align}\label{6.5}
		\varpi^{-1}> {\textstyle \frac{1}{12} (d - 1)^2 d (3 d^2 + d + 10)},
	\end{align} 
	which we will assume henceforth. 
		
	Meanwhile, to control the second and third term in \eqref{5.16} we require that 
	\begin{align}\label{6.6}
		\frac{\frac13(4 d^3 - 19 d + 15)\psi}{1-\Sigma-\sigma } + \frac{\frac12(3 d^2 - 7 d + 4)\psi }{\sigma}< k_2\theta_2 < \frac{1- \frac12(6d^3-11d^2+d+4)\psi }{(2d+1)\sigma},
	\end{align}
	while simultaneously the bound \eqref{4.16} should be satisfied.
	Upon re-writing, we see that the interval in \eqref{6.6} is non-empty when 
	\begin{align}\label{6.7}
		\left(1+\frac{ \frac13 (8 d^4 + 4 d^3 - 38 d^2 + 11 d + 15)\psi}{1  - (6 d^3 - 11 d^2 + d + 4) \psi}\right)\sigma + \Sigma<1.
	\end{align}	 
	When $\psi$ satisfies \eqref{6.3} one can show for $d \ge 5$ that 
	\begin{align*}
		\frac{ \frac13 (8 d^4 + 4 d^3 - 38 d^2 + 11 d + 15)\psi}{1  -  (6 d^3 - 11 d^2 + d + 4) \psi}\le 8,
	\end{align*}
	and hence \eqref{6.7} may be simplified to $9\sigma+\Sigma<1$. 
	In combination with \eqref{6.2} this delivers the bound $s-\rho > b_1(\varpi)$ where
	\begin{align*}
		b_1(\varpi)  = \frac{2^d (d^2+5d-3)}{1 - \frac13 d(d-1)^2(d+13)\varpi}.
	\end{align*}	
	
	In order to handle the bound \eqref{4.16} one confirms that $\delta_2/(1-\sigma-\Sigma)$ is smaller than the first term on the left hand side of \eqref{6.6}, and hence \eqref{4.16} is compatible with the right hand side of \eqref{6.6} if the inequality
	\begin{align*}
		\psi\varpi^{-1} <\frac{1- \frac12(6d^3-11d^2+d+4)\psi }{(2d+1)\sigma}
	\end{align*}	
	is satisfied. Re-arranging yields
	\begin{align*}
		\frac{\psi(2d+1)\varpi^{-1}}{1-\frac12(6d^3-11d^2+d+4)\psi  }\sigma <1,
	\end{align*}	
	which upon inserting \eqref{6.2} delivers the bound $s-\rho > b_2(\varpi, \psi)$ where
	\begin{align*}
		b_2(\varpi, \psi)
		&= \frac{2^d(d-1)(2d+1) \psi }{\varpi (1-  ( d^4 + \frac32 d^3 - \frac{11}{2} d^2 + d + 2)\psi )}.
	\end{align*}
	
	Thus, altogether we have shown that the conclusion of the theorem follows if for some suitable value of $\varpi$ one has
	\begin{align*}
		s-\rho > \max\{a_0(\varpi), a_1(\varpi), a_2(\varpi, \psi), b_1(\varpi), b_2(\varpi, \psi) \}.
	\end{align*}
	We see that  $b_2(\varpi, \psi)>a_2(\varpi, \psi) $ for all admissible values of $\psi$ and $\varpi$. In a similar manner, when $d \ge 5$ a modicum of computation confirms that $a_1(\varpi) \ge \max\{a_0(\varpi) ,b_1(\varpi)\}$ for all admissible values of $\varpi$. 	
	One can compute (for instance with the help of a computer algebra programme) that $a_1(\varpi) = b_2(\varpi, \psi)$ when $\varpi = \varpi_0(\psi)$, where
	\begin{align*}
		 \varpi_0(\psi)=\frac{(d-1)(2d+1)\psi}{d^3 + \frac12d^2 - \frac{11}{2} d + 10 - \psi p_6(d)}.
	\end{align*}
	This quantity is increasing in $\psi$, and a final computation confirms that it is admissible within \eqref{6.5} for all values of $\psi$ satisfying \eqref{6.3}. Thus, for any given value of $\psi$ within the admissible range the bound $s-\rho > b_2(\psi, \varpi_0(\psi))$ dominates overall. Setting $ n_1(\psi)= b_2(\psi, \varpi_0(\psi))$ concludes the proof of Theorem~\ref{T6.1}.
\end{pf}

Theorem~\ref{T1.2} is a simplification of Theorem~\ref{T6.1}. Indeed, upon choosing $\psi = \psi_1$ with $\psi_1^{-1} =2d^4$ we find that 
\begin{align*}
	n_1(\psi_1)&=\frac{2^d( 24 d^7 - 38 d^6 + 39 d^5 + 152 d^4 - 517 d^3 + 732 d^2 - 8 d - 240) } { 24 d^4 - 36 d^3 + 132 d^2 - 24 d - 48} < 2^{d}d(d^2-1)
\end{align*}	
for all admissible values of $d$. Since the function $n_1(\psi)$ is increasing in $\psi$, this bound is sufficient for all $\psi < \psi_1$ also. This completes the proof of Theorem~\ref{T1.2}.\medskip


In order to obtain an estimate for $N_{\cal U}(X, X^\psi)$ and thus complete the proof of Theorems \ref{T1.3} and \ref{T1.4}, we need to sum over all values of $\B y \in \cal U(\Z)$ satisfying $|\B y| \le X^\psi$ and $F(\B y)=0$. 

\begin{thm}\label{T6.2}
	Let $F \in \Z[x_1, \ldots, x_n]$ be a non-singular form of degree $d \ge 5$ defining a hypersurface $\cal V$. Let further $\psi>0$ be a parameter satisfying 
	\begin{align}\label{6.9}
		\psi^{-1}>{\textstyle d^4 + \frac32 d^3 - 5 d^2 + \frac12 d + 2}.
	\end{align} 
	Set 
	\begin{align*}
		n_2(\psi) &=\frac{ 2^{d-1} \left(d^3 + \frac12d^2 - \frac{11}{2} d + 10 - q_6(d) \psi\right) }{1-(d^4 + \frac12 d^3 - \frac52 d^2 - 2 d + 2)\psi}
	\end{align*}
	where $q_6(d)=\frac{1}{12}(50 d^6 - 165 d^5 + 85 d^4 + 481 d^3 - 639 d^2 - 52 d + 240)$.
	For some integer $\rho$ in the range $\frac12d(d+1)+1 < \rho <n$ suppose that $	n-\rho> n_2(\psi)$.
	Then there exists a positive real number $\nu$ for which we have the asymptotic formula
	\begin{align}\label{6.10}
		N_{\cal V_{2, \rho}}(X, Y) = X^{n-D} \sum_{\substack{\B y \in \cal V_{2, \rho}(\Z) \\ |\B y| \le Y}}  \fr S_{\B y} \fr J_{\B y} + O((XY)^{n-D}X^{-\nu}),
	\end{align}
	and the factors satisfy $0 \le \fr S_{\B y} \ll_{\B y} 1$ and $0 \le \fr J_{\B y} \ll_{\B y} 1$.
\end{thm}
\begin{proof}
	Recall from Birch's theorem \cite{birch} that for $n > 2^d(d-1)$ the number of points $\B z \in \Z^n$ with $|\B z| \le Z$ and $F(\B z)=0$ is given by $N(Z) \ll Z^{n-d}$. 
	Upon combining \eqref{1.2} and Proposition~\ref{P5.4}, we find that
	\begin{align}\label{6.8}
		N_{\cal U}(X, X^\psi) = X^{n-D} \sum_{\substack{\B y \in \cal U(\Z) \cap \cal A(\rho) \\ |\B y| \le X^\psi\\F(\B y)=0}}  \fr S_{\B y} \fr J_{\B y} + O\left(E_{\cal A}(\psi)+E_{\cal B}(\psi)+E_{\cal U}(\psi)\right),
	\end{align}
	where
		\begin{align*}
	E_{\cal A}(\psi) &= X^{n-D}\sum_{\substack{\B y \in \cal U(\Z) \cap \cal A(\rho) \\ |\B y| \le X^\psi\\F(\B y)=0}} E(\B y, \theta), & E_{\cal B}(\psi) &= \sum_{\substack{\B y \in \cal U(\Z) \cap \cal B(\rho) \\0 <|\B y| \le X^\psi\\ F(\B y)=0  }} \sum_{\substack{\B x \in \cal U(\Z)  \\ |\B x| \le X \\ F(\B x)=0 }}1
	\end{align*}
	and
	\begin{align*}
		E_{\cal U}(\psi) = \sum_{\substack{ |\B y| \le X^\psi\\F(\B y)=0}} \sum_{\substack{\B x \in \cal V(\Z) \setminus \cal U(\Z) \\ |\B x| \le X}} 1 \ll X^{\dim \cal V \setminus \cal U + \psi(n-d)}.
	\end{align*}
	The choice $\cal U = \cal A(\rho) \cap \cal V = \cal V_{2,\rho}$ entails that $\dim \cal V \setminus \cal V_{2,\rho} = \dim \cal V^*_{2, \rho} \le n-\rho$, and we conclude that the error $E_{\cal U}(\psi)$ is acceptable within \eqref{6.8} if $\rho > D + \psi(D-d)$. In particular, it follows from \eqref{2.10} that the choice $\rho=D+1$ is permissible. Clearly, with this choice of $\cal U$ the  set $\cal B(\rho) \cap \cal U$ is empty and we can disregard the error term $E_{\cal B}(\psi)$.  Thus, it suffices to bound the error $E_{\cal A}(\psi)$. We have
	\begin{align*}
    	E_{\cal A}(\psi) &\ll X^{n-D} N(X^{\psi})\sup_{ |\B y| \le X^\psi} E(\B y, \theta) \ll (X^{1+\psi})^{n-D}(U_1+ U_2 + U_3),
	\end{align*}
	where
	\begin{align*}
      	U_1 &= \sum_{j=2}^d X^{-\delta_j+\psi(\Delta_j+D_{d-j}+ (D_{d-j+2}-2)(d-1)+D-d)-\nu},\\
      	U_2 &= X^{-1+(2d+1)\omega + (3d^3-5d^2+2)\psi }, \nonumber\\
      	U_3 &= X^{{\textstyle \frac{1-\sigma -\Sigma}{\sigma}}(-\omega + \frac12 (3 d^2 - 7 d + 4)\psi )  +\frac16(8 d^3 + 3 d^2 - 41 d + 30)\psi}.
	\end{align*}
	Assuming that
	\begin{align*}
    	\delta_j = \psi(\Delta_j+D_{d-j}+ (D_{d-j+2}-2)(d-1)+D-d) \qquad (2 \le j \le d),
	\end{align*}
	the exponent in the first term is negative. With this choice we have $\delta_d = (D-1) \psi$ and 
	\begin{align*}
		\delta_j \le \delta_2=\textstyle{\frac16} (4 d^3 + 3 d^2 - 25 d + 18)\psi.
	\end{align*} 
	As before, this choice allows us to simplify the conditions \eqref{4.6} and \eqref{4.14}, and we see that they and \eqref{4.15} are satisfied whenever $s-\rho> \max\{a_0(\varpi), a_1(\varpi), a_2(\varpi, \psi)\}$, with the same values as in the proof of Theorem \ref{T6.1}. 

	Meanwhile, the error terms $U_2$ and $U_3$ are acceptable if we can choose $\theta_2$ such that
	\begin{align}\label{6.11}
		\frac{\frac16(8 d^3 + 3 d^2 - 41 d + 30)\psi}{1-\sigma - \Sigma} +	\frac{ \frac12 (3 d^2 - 7 d + 4)\psi}{\sigma}< k_2 \theta_2< \frac{1-  (3d^3-5d^2+2)\psi}{(2d+1)\sigma},
	\end{align}
	and this interval can be seen to be non-empty if \eqref{6.9} is satisfied and further
	\begin{align}\label{6.12}
		\left(1+ \frac{\frac16 (16 d^4 + 14 d^3 - 79 d^2 + 19 d + 30)\psi}{1-\frac12 (12 d^3 - 21 d^2 +  d + 8)\psi}\right)\sigma + \Sigma < 1. 
	\end{align}
	When $\psi$ satisfies \eqref{6.9} one can show for $d \ge 5$ that 
	\begin{align*}
		\frac{\frac16 (16 d^4 + 14 d^3 - 79 d^2 + 19 d + 30)\psi}{1-\frac12 (12 d^3 - 21 d^2 +  d + 8)\psi}\le\frac{25}{3},
	\end{align*}
	and hence \eqref{6.12} can be simplified to $\frac{28}{3}\sigma+\Sigma<1$. Upon recalling \eqref{6.2} this gives $s-\rho> \beta_1(\varpi)$, where
	\begin{align*}
		\beta_1(\varpi) =  \frac{2^d (d^2+\frac{16}{3}d-\frac{10}{3})}{1 - \frac16 d(d-1)^2(2d+27)\varpi}.
	\end{align*}
	
	It remains to compare the right hand side of \eqref{6.11} with the bound of \eqref{4.16}.  As before, with our choice of $\delta_2$ we find that the first term in the maximum in \eqref{4.16} is bounded above by the left hand side of \eqref{6.11}. Thus, it suffices to ensure that the interval 
	\begin{align*}
		\psi \varpi^{-1} < k_2 \theta_2< \frac{1- (3d^3-5d^2+2)\psi}{(2d+1)\sigma}
	\end{align*}
	is non-empty. Such is the case when
	\begin{align*}
		\frac{(2d+1)\psi}{\varpi(1-(3 d^3 - 5 d^2 + 2)\psi)}\sigma<1,
	\end{align*}
	and on inserting \eqref{6.2} we obtain the bound $s-\rho > \beta_2(\varpi, \psi)$ where
	\begin{align*}
		\beta_2(\varpi, \psi)= \frac{2^d (d-1)(2d+1) \psi }{\varpi (1- (d^4 + \frac32 d^3 - 5 d^2 + \frac12 d + 2)\psi)}.
	\end{align*}
	When $d \ge 5$ one checks by a modicum of computation that $\beta_2(\varpi, \psi) \ge a_2(\varpi, \psi)$ and that $a_1(\varpi)$ exceeds both $\beta_1(\varpi)$ and $a_0(\varpi)$ in the appropriate ranges for $\varpi$ and $\psi$. Just as before, we see that $a_1(\varpi) = \beta_2(\varpi, \psi)$ when $\varpi=\varpi_1(\psi)$, where 
	\begin{align*}
		\varpi_1(\psi) = \frac{2(1+2d)(d-1)\psi}{d^3 + \frac12d^2 - \frac{11}{2} d + 10 -  q_6(d)\psi }.
	\end{align*}
	This is in accordance with \eqref{6.5}, so that just as before we obtain our final bound $s-\rho > n_2(\psi)$ where $n_2(\psi)= \beta_2(\varpi_1(\psi), \psi)$. This completes the proof of the theorem.	 
\end{proof}
As before, one can show that $n_2(\psi)$ is increasing in $\psi$, and by taking $\psi=\psi_1$ with $\psi_1^{-1} = 2d^4$ we see after some calculations that 
\begin{align*}
	n_2(\psi_1) &=\frac{2^d(24 d^7 - 38 d^6 + 33 d^5 + 155 d^4 - 481 d^3 + 639 d^2 + 52 d - 240)}{24 d^4 - 36 d^3 + 120 d^2 - 12 d - 48} \\
	&\le 2^d d(d^2-1) - {\textstyle \frac12 d(d+1)}-1.
\end{align*}
The conclusion of Theorem~\ref{T1.3} now follows upon choosing $\rho = \frac12 d(d+1)+1$.

It thus remains to evaluate the sum over the singular integral and singular series. This task can be absolved swiftly by invoking Theorem~2.1 in \cite{FRF2} and imitating arguments from \cite[Section 8]{damaris}. For fixed $Y$ we set  $\psi_0 = (d^3(d+\frac32) - 1)^{-1}$ and $X_0 = Y^{1/\psi_0}$. Now assume that 
\begin{align}\label{6.13}
    n - \rho > 2^{d-1}d(d+1)(1+\psi_0^{-1}).
\end{align}
Then by \cite[Theorem~2.1]{FRF2} we have the alternative asymptotic formula
\begin{align*}
    N(X_0, Y) =(X_0Y)^{n-D}  \chi_{\infty}\prod_{p \text{ prime}} \chi_p  + O((X_0Y)^{n-D} Y^{-\nu}).
\end{align*}
On the other hand, one can check that the condition in \eqref{6.13} is stricter than the hypothesis of Theorem~\ref{T6.2}, so we may compare this bound with \eqref{6.10} and deduce that
\begin{align}\label{6.14}
    \sum_{\substack{\B y \in \cal U(\Z) \\ |\B y| \le Y\\F(\B y)=0}} \fr S_{\B y} \fr J_{\B y}= Y^{n-D} \chi_{\infty}\prod_{p \text{ prime}} \chi_p  + O(Y^{n-D-\nu}).
\end{align}
Note in particular that \eqref{6.14} does not depend on $X_0$ any longer. Thus, if \eqref{6.13} is satisfied, we are able to replace the sum over the singular series and integral in in Theorem~\ref{T1.3} by a product of local densities as in \eqref{6.14}. This establishes Theorem~\ref{T1.1} for all $\psi \le \psi_0$, while for $\psi_0 \le \psi \le 1$ the corresponding result follows from Theorem~2.1 in \cite{FRF2}. Finally, we recall that we need $\rho \ge \frac12d(d+1)+1$ and note that 
\begin{align*}
	{\textstyle2^{d-1}d^4(d+1)(d+ \frac32) + \frac12 d(d+1)+1 }\le 2^{d-1}d^4(d+1)(d+2)
\end{align*}  
for all admissible values $d$. This completes the proof of Theorem~\ref{T1.1}.

In order to complete the proof of our final result in Theorem~\ref{T1.4}, we note that in this case $\cal U = \cal V \setminus \{ \bm 0\}$. Thus, the error $E_{\cal U}(\psi) \ll X^{\psi(n-d)}$ is under control, and it remains to understand the error arising from any singular set $\cal B(\rho)$. From \eqref{3.9} we infer that 
$	E_{\cal B}(\psi) \ll X^{n-d} X^{\psi(n-\rho)}, 
$ 
which is acceptable within \eqref{6.8} if
$
	\rho > D +  d(d-1)/(2\psi).
$ 
Picking $\rho$ minimal in this way, we can now proceed precisely as in the proof of Theorem~\ref{T1.1}.

\end{document}